\pdfoutput=1
\documentclass{amsart} 

\usepackage[letterpaper,
    outer=3.1cm,
    inner=3.1cm,
    bottom=3.25cm,
    top=3.25cm]{geometry} 

\usepackage{amsmath, amsxtra, amsthm, amsfonts, amssymb, mathtools}
\usepackage{manfnt}
\usepackage{mathrsfs}
\usepackage{graphicx}
\usepackage{url}
\usepackage{color} 
\usepackage{bm}
\usepackage{tikz-cd}
\usepackage{tikz}
\usetikzlibrary{knots}
\usepackage{tikz-3dplot}
\usepackage{tkz-euclide}
\usepackage[all,cmtip]{xy}	
\xyoption{arrow}

\usepackage[T1]{fontenc}
\usepackage{enumitem}
\usepackage{multicol}

 \usepackage[hidelinks, backref]{hyperref} 
 \hypersetup{
    colorlinks,
    citecolor=magenta,
    filecolor=magenta,
    linkcolor=blue,
    urlcolor=magenta
}
\usepackage[capitalise]{cleveref}

\theoremstyle{definition}
\newtheorem{thm}[subsection]{Theorem}

\newtheorem{proposition}[subsection]{Proposition}

\newtheorem{corollary}[subsection]{Corollary}
\newtheorem{lemma}[subsection]{Lemma}
\newtheorem{definition}[subsection]{Definition}
\newtheorem{example}[subsection]{Example}
\newtheorem{remark}[subsection]{Remark}

\newtheorem*{thm:multimatroid_integral}{Theorem~\ref{thm:multimatroid_integral}}

\newtheorem{innercustomthm}{Theorem}
\newenvironment{customthm}[1]
  {\renewcommand\theinnercustomthm{#1}\innercustomthm}
  {\endinnercustomthm}

\def\Z{\mathbb{Z}}

\def\R{\mathbb{R}}

\def\P{\mathbb{P}}
\def\A{\mathbb{A}}
\def\N{\mathbb{N}}
\def\Z{\mathbb{Z}}

\def\calL{\mathcal{L}}

\def\calP{\mathcal{P}}

\newcommand{\T}{\mathcal{T}}

\def\bfM{\mathbf{M}}

\newcommand{\bfw}{\mathbf{w}}
\newcommand{\e}{\mathbf{e}}
\newcommand{\rkk}{\mathbf{rk}}

\newcommand*\cube{\scalebox{0.7}{\mbox{\mancube}}}
\newcommand{\scC}{\mathscr{C}}
\newcommand{\scM}{\mathscr{M}}
\newcommand{\scMc}{\mathscr{M}^{\cube}}
\newcommand{\coeff}{c}

\newcommand{\aaS}{\vec{a}}

\newcommand{\NC}{\mathrm{C}}

\newcommand{\Mbar}{\overline{\mathcal{M}}}
\newcommand{\decoset}{\mathcal{R}_\pi}
\newcommand{\decoseto}{\mathcal{R}_\pi^\times}
\newcommand{\decosetmax}{\mathcal{R}_\pi^{\max}}
\newcommand{\mchain}{\mathsf{MaxChain}}
\newcommand{\maxchain}{\mathsf{MaxChain}(\decoset)}

\newcommand{\ove}{\overline{\mathbf{e}}}
\newcommand{\cone}{\operatorname{cone}}
\newcommand{\conv}{\operatorname{conv}}

\newcommand{\Hom}{\operatorname{Hom}}
\newcommand{\IP}{\textrm{IP}}
\newcommand{\IPC}{\textrm{IPC}}

\newcommand{\Lrn}{\overline{\calL}^{r}_n}
\newcommand{\Span}{\operatorname{span}}

\newcommand{\M}{\mathbf{M}}
\renewcommand{\emptyset}{\varnothing}

\renewcommand{\L}{\overline{\mathcal{L}}}
\newcommand{\rk}{\textrm{rk}}
\newcommand{\vol}{\operatorname{Vol}}
\newcommand{\Vol}{\operatorname{Vol}}
\newcommand{\blu}[1]{\textcolor{black}{#1}}
\newcommand{\blub}[1]{\textcolor{black}{\overline{#1}}}
\newcommand{\red}[1]{\textcolor{black}{#1}}
\newcommand{\redb}[1]{\textcolor{black}{\overline{#1}}}

\newcommand{\annoyingfanname}{affine permutohedral fan}

\definecolor{ao(english)}{rgb}{0.5, 0.0, 0.5}
\definecolor{ao(green)}{rgb}{0, 0.5, 0}

\usepackage{xargs}
\usepackage[colorinlistoftodos, prependcaption ,textsize=small
]{todonotes}

\title{Multimatroids and rational curves with cyclic action}
\author[E. Clader]{Emily Clader}\address{Emily Clader, Department of Mathematics, San Francisco State University}
\email{\url{eclader@sfsu.edu}}

\author[C. Damiolini]{Chiara Damiolini}
\address{Chiara Damiolini, Department of Mathematics, University of Texas at Austin}
\email{\url{chiara.damiolini@austin.utexas.edu}}

\author[C. Eur]{Christopher Eur}
\address{Christopher Eur, Department of Mathematics, Harvard University}
\email{\url{ceur@math.harvard.edu}}

\author[D. Huang]{Daoji Huang}\address{Daoji Huang, Department of Mathematics, University of Minnesota}
\email{\url{huan0664@umn.edu}}

\author[S. Li]{Shiyue Li}\address{Shiyue Li, School of Mathematics, Institute for Advanced Study}
\email{\url{shiyue_li@ias.edu}}

\begin{document}

\begin{abstract}
We study the connection between multimatroids and moduli spaces of rational curves with cyclic action. Multimatroids are generalizations of matroids and delta-matroids that naturally arise in topological graph theory.  The perspective of moduli of curves provides a tropical framework for studying multimatroids, generalizing the previous connection between type-$A$ permutohedral varieties (Losev--Manin moduli spaces) and matroids, and the connection between type-$B$ permutohedral varieties and delta-matroids.
Specifically, we equate a combinatorial nef cone of the moduli space with the space of $\R$-multimatroids, a generalization of multimatroids, and we introduce the independence polytopal complex of a multimatroid, whose volume is identified with an intersection number on the moduli space.  As an application, we give a combinatorial formula for a natural class of intersection numbers on the moduli space by relating to the volumes of independence polytopal complexes of multimatroids.
\end{abstract}

\maketitle

\section{Introduction}

There have been rapid recent developments in the interplay amongst three objects: Coxeter groups, matroids, and the Chow rings of certain moduli spaces of rational curves.  In type $A$, the key insight is that the base polytope of a matroid on a set with $n$ elements is a type-$A$ generalized permutohedron \cite{GGMS87}, meaning that its normal fan coarsens the type-$A$ permutohedral fan $\Sigma_{A_{n-1}}$.  This allows one to associate to any matroid an element of the Chow ring of the toric variety $X_{A_{n-1}}$, which was realized by the work of Losev and Manin \cite{LM} as a moduli space of rational curves with weighted marked points.  The connections between these perspectives have yielded breakthroughs in both matroid theory and geometry \cite{AHK, BST, BEST, DR, EHL, Hampe, LdMRS}.

For type-$B$ Coxeter groups, a similar unifying framework was developed in \cite{EFLS, eur2023k}, establishing a connection between the algebraic geometry of the type-$B$ permutohedral fan $\Sigma_{B_n}$ and the combinatorics of delta-matroids, an analogue of matroids first introduced by Bouchet \cite{bouchet1987}.  Batyrev and Blume showed that the toric variety $X_{B_n}$ also admits a modular interpretation as a moduli space of rational curves equipped with an involution \cite{BB1, BB2}. 

At present, it seems that this story does not extend to other Coxeter types.  In particular, obstacles were encountered while studying the tautological classes of other Coxeter matroids \cite[Remark 3.6]{EFLS} and in finding a modular interpretation for toric varieties corresponding to other Coxeter types \cite{BB1}.

On the other hand, there is another family of complex reflection groups that generalize type-$A$ and type-$B$ Coxeter groups, called generalized symmetric groups $S(r, n)$, which depend on parameters $r \geq 2$ and $n \geq 1$.  In our previous work \cite{clader2022permutohedral, clader2022wonderful}, we constructed moduli spaces of curves that correspond to these groups in a precise sense.  The result is a smooth projective moduli space $\Lrn$ parameterizing rational stable curves with an order-$r$ automorphism and $n$ orbits of weighted points, which coincides when $r=2$ with Batyrev--Blume's space.  A similar generalization applies in matroid theory: delta-matroids are the $r=2$ case of objects known as $r$-matroids, which are a special case of multimatroids \cite{bouchet1987}.  The connection between the theory of $r$-matroids and the geometry of $\Lrn$ has not yet been studied, and developing this connection is our primary motivation for the current work.

To describe the more general setting in which we work, fix a positive integer $n$, a finite set $E$, and a surjection $\pi \colon  E \rightarrow [n]$, where $[n] \coloneqq \{1, \ldots, n\}$.  The data of $\pi$ is equivalent to a partition
\[E = E_1 \sqcup \cdots \sqcup E_n\]
by setting $E_i \coloneqq  \pi^{-1}(i)$ for each $i$.  A subset $S \subseteq E$ is {\bf $\pi$-colored} if it contains at most one element of each $E_i$.  We denote by $\decoset$ the poset of $\pi$-colored subsets of $E$, ordered by inclusion.  Note that the maximal elements of this poset are the size-$n$ subsets of $E$ consisting of precisely one element of each $E_i$.

In the same way that a matroid on ground set $E$ can be defined via a rank function on subsets of $E$, a {\bf multimatroid} is a rank function $\rkk: \decoset \rightarrow \N$ that satisfies analogous axioms specified in \cref{def:multimatroid}.  A key feature of these axioms is that, for any maximal $\pi$-colored subset $T \subseteq E$, the restriction of $\rkk$ to subsets of $T$ (all of which are automatically $\pi$-colored) is a matroid in the usual sense; in this way, a multimatroid can roughly be viewed as a way of patching together a collection of matroids on equal-sized ground sets.

The role played by the permutohedral fan in the theory of matroids is played in the theory of multimatroids by the \textbf{$\pi$-colored fan} $\Sigma^\pi$, which we introduce in \cref{sec:fan}; it is the $n$-dimensional fan in the vector space
\[N^\pi_{\R} \coloneqq \R^{E_1}/\R \mathbf{1} \times \cdots \times \R^{E_n}/\R\mathbf{1}\]
with a cone
\[\sigma_{\scC} = \cone\{\ove_{S_1}, \ldots, \ove_{S_k}\}\]
for each chain $\scC = (S_1 \subsetneq \cdots \subsetneq S_k)$ of nonempty $\pi$-colored subsets of $E$, where $\ove_S$ denotes the image in $N^\pi_\R$ of $\sum_{i \in S} \e_i \in \R^E$.  For any maximal $\pi$-colored subset $T \subseteq E$, the intersection of $\Sigma^\pi$ with the subspace $\R_{\geq 0} \cdot \{\ove_i \; | \; i \in T\} \subseteq N^\pi_\R$ is identical to a distinguished orthant (which we call the {\bf\annoyingfanname}) of the stellahedral fan studied in \cite{EHL}.  So, analogously to the above perspective on multimatroids, the $\pi$-colored fan can roughly be viewed as a way of patching together a collection of {\annoyingfanname}s; see Remarks \ref{rem:M(T)fan} and \ref{rem:M(T)polytope}.

Toric geometry allows one to give an explicit presentation of the Chow ring of the toric variety $X_{\Sigma^\pi}$ as a quotient of
\[
\Z[x_S \; | \; S \in \decoset]
\] 
with relations described explicitly in \cref{prop:Chowring}.  In the special case where $|E_i| = r$ for each $i$ (in which case we refer to $\pi$ as a {\bf uniform} partition), the $\pi$-colored fan $\Sigma^{\pi}$ coincides with the fan $\Sigma^r_n$ studied in \cite{clader2022wonderful}, and the results of that work show that
\begin{equation}
    \label{eq:LrnSigmapi}
A^*(\Lrn) \cong A^*(\Sigma^r_n).
\end{equation}
The perspective on $\Sigma^r_n$ as a union of {\annoyingfanname}s can be given a precise geometric interpretation in this setting, as explained in \cref{rem:M(T)geometric}.

One crucial property of $\Sigma^{\pi}$ is that it is a balanced tropical fan, and in particular, by \cite[Proposition 5.6]{AHK}, there  exists a well-defined degree map 
\[
\int_{\Sigma^{\pi}} \colon A^n(\Sigma^{\pi}) \to \R
\] with the defining property that $\int_{\Sigma^{\pi}} \sigma = 1$ for each maximal  cone $\sigma$.  Our first main result is a combinatorial formula for these degrees.

Such a formula can be stated in terms of the generators $x_S$ of $A^*(\Sigma^\pi)$, and in some sense, these are the most geometrically natural generators: they correspond to the rays of $\Sigma^\pi$ and, in the uniform case, they are identified by the isomorphism \eqref{eq:LrnSigmapi} with the boundary divisors of $\Lrn$.  However, as was already understood in the matroid setting by \cite{BES}, the formula becomes much more combinatorially elegant when stated in terms of a different basis.  Specifically, for each $S \in \decoset$, set
\[h_S \coloneqq \sum_{S' \cap S \neq \emptyset} x_{S'}.\]
One way to understand the special role played by these alternative generators (analogously to \cite{DR}) is that, in the uniform case, they can be viewed under the isomorphism \eqref{eq:LrnSigmapi} as pullbacks of psi-classes under a family of forgetful morphisms from $\Lrn$ to a simpler moduli space, and the combinatorially rich structure of psi-classes is well-understood; see \cref{sec:psiclasses}. 

To state the formula for the degree of a monomial in the above generators of $A^*(\Sigma^\pi)$, we define, for any collection $S_1, \ldots, S_n$ of $\pi$-colored subsets (possibly with repetitions), the set
\[\T_\pi(S_1, \ldots, S_n) \coloneqq  \left\{ T \in \decosetmax \; \left| \; \begin{matrix} \text{there exists a bijection } \iota:[n] \rightarrow T\\ \text{with } \iota(i) \in S_i \text{ for each }i \end{matrix} \right.\right\}.\] In other words, $\T_\pi(S_1, \ldots, S_n)$ consists of the maximal elements of $\decoset$ containing precisely one element from each of the sets $S_i$. Then we have the following formula.

\begin{customthm}{A}
\label{thm:h-integrals}
For any collection $S_1, \ldots, S_n \in \decoset$ (with repetitions allowed), we have
\[\int_{\Sigma^{\pi}} h_{S_1} \cdots h_{S_n} = |\T_\pi(S_1, \ldots, S_n)|.\]
\end{customthm}

The analogue of this theorem in type $A$ follows from \cite[Theorem 5.1]{Postnikov} and \cite{BES}, while the type-$B$ case is proved in \cite[Theorem A(b)]{EFLS}.  At the same time, given the perspective on $h_S$ as a pullback of a psi-class in the uniform case, this theorem can be viewed as an analogue of the computations of intersection numbers of psi-classes on $\overline{\mathcal{M}}_{0,n}$ in \cite{BELL, Witten}.  This suggests that \cref{thm:h-integrals} may admit an algebro-geometric proof via the theory of psi-classes; we investigate this direction in Section~\ref{sec:psiclasses}, but we do not currently know of a complete proof by these methods.

Instead, we prove \cref{thm:h-integrals} via the geometry and combinatorics of multimatroids.  To do so, we introduce the independence polytopal complex $\IPC(\bfM)$ of a multimatroid in  \cref{def:IPC}, generalizing the independence polytope $\IP(M)$ of a matroid $M$.  We prove that the volume of this independence polytopal complex, when suitably normalized (see \cref{sec:IPC}), coincides with the degree of the top power of a divisor on $\Sigma^\pi$ naturally associated to $\bfM$.  Specifically, for any multimatroid $\bfM$, let
\[D_\bfM \coloneqq \sum_{S \in \decoset} \rkk(S) x_S \in A^1(\Sigma^\pi).\]
Then we have the following theorem. 

\begin{customthm}{B} \label{thm:B}
 (see \cref{thm:multimatroid_integral})
For any multimatroid $\bfM$ on $(E, \pi)$,
\begin{equation} \label{eq:multimatroid-integral}\tag{B} \int_{\Sigma^\pi} \left( D_\bfM\right)^n = \vol(\IPC(\bfM)). \end{equation}
\end{customthm}

\begin{remark}
Matroids are a special case of multimatroids, and the analogue of \cref{thm:B} in this setting is the equality, for any matroid $M$, of the degree of the divisor $D_M$ on the stellahedral fan and the volume of the independence polytope of $M$.  In this setting, the theorem can be deduced from a standard result in toric geometry: the volume of the polytope corresponding to a nef divisor $D$ on a rational, complete fan $\Sigma$ is equal to $\int_{\Sigma} D^n$ (see, for example \cite[page 111]{FultonToric}).  The divisor $D_M$ is known to be nef, and its corresponding polytope is precisely the independence polytope.  The multimatroid case, on the other hand, is considerably more subtle. 
\end{remark}

\Cref{thm:B} implies \cref{thm:h-integrals} via the results of \cite{EL}, as we show in \cref{sec:proofs}.  In order to prove \cref{thm:B}, the main idea is to use the work of Nathanson--Ross \cite{NR23}, which relates the degrees of top powers of divisors on tropical fans to the volumes of associated polytopal complexes known as normal complexes.  However, their results generally apply only to divisors satisfying a cubical condition, which $D_\bfM$ does not necessarily satisfy. We resolve this obstruction by extending the statement as an equality of functions on the space of multimatroids on $(E,\pi)$. 

The key idea is to consider a slight generalization of the notion of multimatroid that we refer to as an {\bf $\R$-multimatroid}, which consists of a rank function $\rkk\colon \decoset \rightarrow \R$ satisfying the properties listed in \cref{def:multimatroid-rank}. The notions of $D_\bfM$ and $\IPC(\bfM)$ extend to this setting, so the statement of \cref{thm:B} makes sense when $\bfM$ is an $\R$-multimatroid, and it is in this setting that we prove the theorem.
The advantage of this extension is two-fold:
\begin{enumerate}
    \item The space $\scM$ of $\R$-multimatroids on $(E,\pi)$ is a connected subspace of $\R^{\decoset}$, while the space of all multimatroids on $(E,\pi)$ is discrete (see \cref{sec:scM}).
    \item For a given $(E,\pi)$, one can always find an $\R$-multimatroid $\bfM$ for which $D_\bfM$ is cubical (\cref{lem:non-empty}); this is not true if we only consider multimatroids. 
\end{enumerate}

Given this extension, we prove \cref{thm:multimatroid_integral} (and hence  \cref{thm:B}) by showing that both sides of \eqref{eq:multimatroid-integral} are polynomial functions on $\scM$ that agree---via the work of \cite{NR23}---on the subset consisting of $\R$-multimatroids $\bfM$ for which $D_\bfM$ is cubical.  Since this locus is non-empty and open (\cref{cor:cubical-open}), this implies that the two polynomial functions agree on all of $\scM$, showing that the theorem holds for every $\R$-multimatroid.

\subsection{Future directions}
One of the reasons for our interest in \cref{thm:h-integrals} is that it provides evidence for the existence of an exceptional isomorphsim from the Chow ring $A(\Lrn)$ to the Grothendieck $K$-ring of vector bundles $K(\Lrn)$ similar to isomorphisms appearing in the study of matroids and delta-matroids \cite{BEST, EFLS, EHL, LLPP}.  In future work, we plan to study the conjectural existence of such an isomorphism, which we hope will yield a Hirzebruch--Riemann--Roch-type formula for computing Euler characteristics of vector bundles on $\overline{\calL}^{r}_n$.  In the case of matroids and delta-matroids, this isomorphism also relates to an isomorphism with the polytope algebra of generalized (type-$A$ or type-$B$) permutohedra, so we hope along the way to relate the $K$-ring of $\Lrn$ to a polytopal complex algebra.  Such a connection would also yield a relationship between Euler characteristics of vector bundles on $\Lrn$ and lattice point counts of certain polytopal complexes of multimatroids, analogous to the case of matroids \cite{CameronFink} and toric varieties \cite[Section 5.3]{FultonToric}. As applications of this circle of ideas, we hope to apply the present framework to the study of certain polynomials of multimatroids, embedded graphs, and knots \cite{ellis2013graphs}.

\subsection{Outline of the paper}

Section \ref{sec:coloredfan} introduces the fan $\Sigma^\pi$ and the relevant generators of its Chow ring.  In Section \ref{sec:normalcomplexes}, we first review from \cite{NR23} the definition of the normal complex $C_{\Sigma, \ast}(D)$ of a fan $\Sigma$ equipped with a divisor $D \in A^1(\Sigma)_\R$, as well as the relation between the volume of the normal complex and $\int_\Sigma D^n$ under the condition that $D$ is cubical.  We then specialize this framework to the case of $\Sigma^\pi$, in which case we can make both the notion of volume and the cubical condition  concrete.  We turn in \cref{sec:multimatroids} to the definition of multimatroids and $\R$-multimatroids, and we explain how to associate to any $\R$-multimatroid $\bfM$ both a divisor $D_{\bfM} \in A^1(\Sigma^\pi)_\R$ and an independence polytopal complex $\IPC(\bfM)$.  The key result of this section is that there is a nonempty open subset in the space of $\R$-multimatroids on which the divisor $D_{\bfM}$ is cubical, and the first result of Section~\ref{sec:proofs} is that $\IPC(\bfM)$ is equal to the normal complex $C_{\Sigma^\pi, \ast}(D_\bfM)$ (with equivalent notions of volume) in this case.  Combining these results with \cite{NR23} proves \cref{thm:B}, and the remainder of \cref{sec:proofs} is devoted to unpacking \cref{thm:B} from the perspective of the generators $h_S$ in order to deduce \cref{thm:h-integrals}.  Lastly, in \cref{sec:psiclasses}, we specialize to the case in which $\pi$ is uniform with $|E_i| = r$ for each $i$, and we use the isomorphism $A^*(\Sigma^\pi) \cong A^*(\Lrn)$ to reprove some cases of \cref{thm:h-integrals} from a geometric perspective.

\subsection*{Acknowledgments}
We are grateful to Federico Ardila, Renzo Cavalieri, Matt Larson, Rohini Ramadas, and Dusty Ross for helpful discussions.  We thank the anonymous referee for their insights, which in particular led to the proof of \cref{lem:(2.9)}. EC is supported by NSF CAREER Grant 2137060.  CE is supported by NSF Grant DMS--2246518.  DH is supported by NSF Grant DMS--2202900. SL is supported by NSF Grant DMS--1926686 and partially supported by a Coline M. Makepeace Fellowship from Brown University Graduate School.

\section{The \texorpdfstring{$\pi$}{pi}-colored fan}
\label{sec:coloredfan}

Throughout what follows, we fix a nonempty finite set $E$ with a partition
\[E = E_1 \sqcup \cdots \sqcup E_n,\]
or equivalently, a surjective map $\pi \colon  E \to [n]$ where $\pi^{-1}(i) = E_i$.  We refer to the partition as {\bf uniform} if $|E_i| = |E_j|$ for all $i,j \in [n]$.

\subsection{Colored sets}  \label{sec:transversal}

Viewing $[n]$ as the set of possible colors, and $\pi$ as a way to assign a unique color to every element of $E$, we are particularly interested in subsets of $E$ that contain at most one element of each color.  More precisely, we have the following definition.

\begin{definition}
    A subset $S \subseteq E$ is {\bf $\pi$-colored} (or just {\bf colored}, if $\pi$ is clear from context) if
        \[|S \cap E_i| \leq 1\]
    for each $i \in [n]$.  We denote by $\decoset$ the poset of colored subsets of $E$, ordered by inclusion.  Maximal elements of $\decoset$ are those that contain exactly one element of each $E_i$, and we denote the set of these by $\decosetmax$.  We often wish to exclude the possibility that $S = \emptyset$, so we denote $\decoseto = \decoset \setminus \{\emptyset\}$.
\end{definition}

\begin{remark} One can generalize the notion of colored sets by requiring that the set $S \cap E_i$ has at most a specified number of elements $c_i$, which might depend on $i \in [n]$.  Though we do not take up this generalization in this work, it would be interesting to investigate the extent to which the results of this paper generalize to that setting.
\end{remark}

\subsection{The \texorpdfstring{$\pi$}{pi}-colored fan} \label{sec:fan}

To define a fan associated to the data of $(E,\pi)$, we consider the real vector space $\R^E$ with standard basis $\{\e_i \; | \; i \in E\}$.  For each $X \subseteq E$, denote
\[\e_X \coloneqq \sum_{i \in X} \e_i.\]
Set
\begin{equation}
    \label{eq:NpiR}
 N^\pi_\R \coloneqq  \dfrac{\R^{E_1}}{\R \e_{E_1}} \times \dots \times \dfrac{\R^{E_n}}{\R \e_{E_n}},
 \end{equation}
and denote the image of $\e_i$ or $\e_X$ in $N^\pi_\R$ by $\ove_i$ or $\ove_X$, respectively. Similarly, for every $X \subseteq E$, denote by $\overline{\R}^{X}$ the image of $\R^{X} \subseteq \R^E$ in $N^\pi_\R$. 

\begin{definition} \label{def:Sigma} The {\bf $\pi$-colored fan} $\Sigma^\pi$ is the fan in $N^\pi_\R$ consisting of cones
\[
\sigma_{\scC} = \cone\{ \ove_{S_1}, \ldots, \ove_{S_k}\}
\] 
for each chain $\scC = \big(S_1 \subsetneq \cdots \subsetneq S_k \big)$ of elements $S_i \in \decoseto$.
\end{definition}

In the special case where $(E,\pi)$ is uniform with $|E_i| = r \geq 2$ for each $i$, the fan $\Sigma^\pi$ coincides with the $r$-permutohedral fan $\Sigma^r_n$ studied in \cite{clader2022wonderful}. If, furthermore, $r=2$, then it is the type-$B$ permutohedral fan $\Sigma_{B_n}$.  We begin by illustrating what $\Sigma^\pi$ looks like in this particularly simple case.

\begin{example}\label{ex:SigmaB2}
Let 
$E=\{{1}, {\bar{1}}\} \sqcup \{ {2}, {\bar{2}}\}$ with $E_1=\{{1}, {\bar{1}}\}$ and $E_2=\{ {2}, {\bar{2}}\}$. 
Then
\[N^\pi_\R = \frac{\R\e_1 \oplus \R\e_{\bar{1}}}{\R(\e_1+\e_{\bar{1}})} \times \frac{\R\e_2\oplus \R\e_{\bar{2}}}{\R(\e_2+\e_{\bar{2}})}.\]
Choosing the basis $\{\ove_{{1}}, \ove_{{2}}\}$ for $N^\pi_\R$ gives an isomorphism $N^\pi_\R \cong \R^2$ in which 
\[\ove_{{1}} \mapsto (1,0), \quad \ove_{\bar{1}} \mapsto (-1,0), \quad \ove_{{2}} \mapsto (0,1), \quad \ove_{\bar{2}} \mapsto (0,-1).\]
The fan $\Sigma^\pi = \Sigma_{B_2}$ is depicted under this isomorphism in Figure~\ref{fig:SigmaB2}.
\end{example}

\begin{figure}[ht]
    \centering
\begin{tikzpicture}[xscale=0.5, yscale=0.5]
\draw [fill=gray!30!white,gray!30!white] (-2.4,-2.4) rectangle (2.4,2.4);
\draw [ao(english),fill=ao(english),opacity=0.3] (0,0) -- (0,2.4) -- (-2.4,2.4)  -- cycle;
\draw [thick,->] (0,0) --(3,0);
\draw [thick,->] (0,0) --(-3,0);
\draw [ao(english),thick,->] (0,0) --(0,3);
\draw [thick,->] (0,0) --(0,-3);
\draw [thick,->] (0,0) --(2.7,2.7)  node[above right] {$\ove_{\{{1},{2}\}}$};
\draw [ao(english),thick,->] (0,0) --(-2.7,2.7)  node[above left] {$\ove_{\{\bar{1},{2}\}}$};
\draw [thick,->] (0,0) --(2.7,-2.7) node[below right] {$\ove_{\{{1},\bar{2}\}}$};
\draw [thick,->] (0,0) --(-2.7,-2.7)  node[below left] {$\ove_{\{\bar{1},\bar{2}\}}$};

\filldraw [ao(english)] (0,0) circle[radius=1mm];
\node at (0, -3.6) {$\ove_{\bar{2}}$};
\node [ao(english)] at (0, 3.6) {$\ove_{{2}}$};
\node at (3.6,0) {$\ove_{{1}}$};
\node at (-3.6,0) {$\ove_{\bar{1}}$};
\node [ao(english)] at (-1.1,2.8) {$\sigma_{\scC}$};
\end{tikzpicture}
\caption{The fan $\Sigma_{B_2}$, with the cone corresponding to $\scC= (\{{2}\} \subsetneq \{\bar{1},{2}\})$ shaded.}
    \label{fig:SigmaB2}
\end{figure}
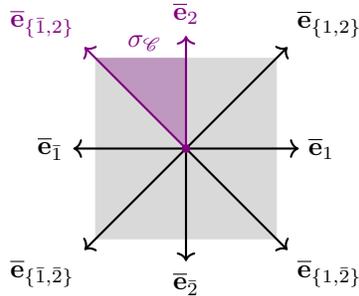

On the other hand, the following example illustrates a non-uniform case.

\begin{example}\label{ex:nonuniform}
    Let $E=\{{a}, {b}, {c}\} \sqcup \{{1}, {2}\}$ with $E_1 = \{a,b,c\}$ and $E_2 = \{1,2\}$.  Then Figure~\ref{fig:SigmaNU} depicts the associated fan $\Sigma^\pi$, under the isomorphism  $N^\pi_\R     \cong \R^3$     given by the basis $\{\ove_a, \ove_b, \ove_1\}$.
    \begin{figure}[h!]
\centering
\tdplotsetmaincoords{30}{35} 
\begin{tikzpicture}[tdplot_main_coords,scale=1.7]
    \coordinate (A) at (0,0,0);
 
    \filldraw [gray!70!white,opacity=0.5] (A) -- (0,1,0) -- (0,1,1)-- (0,0,1) -- cycle;
    \filldraw [gray!70!white,opacity=0.5] (A) -- (0,1,0) -- (0,1,-1) -- (0,0,-1) -- cycle;
    \draw[thick, ->] (0,0,0) -- (0,1.11,1.11);
	\draw[thick, ->] (0,0,0) -- (0,1.11,-1.11);
    \draw[thick, ->] (0,0,0) -- (0,1.11,0) node[above right] {$\ove_{{b}}$};

    \filldraw [gray!70!white,opacity=0.5] (A) -- (-1,-1,0) -- (-1,-1,1) -- (0,0,1) -- cycle;
    \filldraw [gray!70!white,opacity=0.5] (A) -- (-1,-1,0)  -- (-1,-1,-1) -- (0,0,-1) -- cycle;    
    \filldraw [gray!70!white,opacity=0.5] (A) -- (1,0,0) --(1,0,1) -- (0,0,1) -- cycle;
    \filldraw [gray!70,opacity=0.5] (A) -- (1,0,0) -- (1,0,-1) -- (0,0,-1) -- cycle;
    \filldraw [ao(english)!50,opacity=0.7] (A) -- (1,0,0) -- (1,0,-1) -- cycle;
    
	\draw[thick, ->] (0,0,0) -- (0,0,1.21) node[above] {$\ove_{{1}}$}; 
	\draw[ao(english), thick, ->] (0,0,0) -- (1.11,0,-1.11) node[right] {$\ove_{\{{a},{2}\}}$};		
	\draw[thick, ->] (0,0,0) -- (1.11,0,1.11);
	\draw[ao(english), thick, ->] (0,0,0) -- (1.11,0,0) node[right] {$\ove_{{a}}$};
    \draw[thick, ->] (0,0,0) -- (0,0,-1.21) node[below] {$\ove_{{2}}$};

    \draw[thick, ->] (0,0,0) -- (-1.11,-1.11,0) node[left] {$\ove_{{c}}$};

    \draw[thick, ->] (0,0,0) -- (-1.11,-1.11,1.11);
    \draw[thick, ->] (0,0,0) -- (-1.11,-1.11,-1.11);
      
	\filldraw [ao(english)] (0,0,0) circle[radius=0.3mm];    

    \node [ao(english)] at (1.34,-0.2,0) {$\sigma_{\scC}$};
  \end{tikzpicture}
  \caption{The fan $\Sigma^\pi$ for $(E,\pi)$ as in \cref{ex:nonuniform}, with the cone corresponding to $\scC= (\{{a}\} \subsetneq \{{a},{2}\})$ shaded.}
\label{fig:SigmaNU}\end{figure}
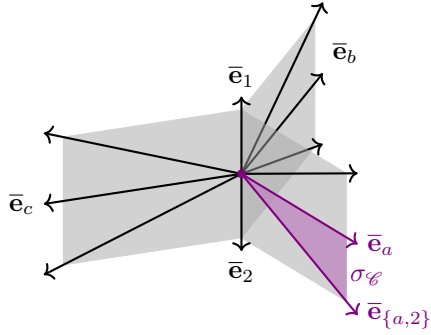 
\end{example}

\begin{remark}
    \label{rem:M(T)fan}
For any $T \in \decosetmax$, the restriction of $\Sigma^\pi$ to the subset $\R^T_{\geq 0} \cong \R^n_{\geq 0}$ is a fan $\Sigma_T$ that can be identified with the fan in $\R^n_{\geq 0}$ consisting of a cone for each chain of subsets of $[n]$.
This fan in $\R_{\geq 0}^n$, which we call the \textbf{\annoyingfanname}, is a distinguished portion (the negative orthant) of the stellahedral fan defined in \cite{BHMPW22}.
Thus, $\Sigma^\pi$ can be viewed as a union of copies of the $n$-dimensional \annoyingfanname, one for each $T \in \decosetmax$.  Given the connection between stellahedral fans and matroids studied in \cite{EHL}, this observation can be seen as the fan-theoretic analogue of the perspective mentioned in the introduction that a multimatroid is a way of patching together a collection of matroids.
\end{remark}

\subsection{The Chow ring of the \texorpdfstring{$\pi$}{pi}-colored fan}

Standard results in toric geometry calculate the Chow ring of $\Sigma^\pi$ (or, equivalently, of the associated toric variety). 

\begin{proposition} 
\label{prop:Chowring}
    The Chow ring of $\Sigma^\pi$ is
    \begin{equation*}
        A^*(\Sigma^\pi) = \frac{\Z[x_S \; | \; S \in \decoseto]}{\mathcal{I} + \mathcal{J}},
    \end{equation*}
    where $\mathcal{I}$ is the ideal of quadratic relations
    \begin{equation}
    \label{eq:quadratic}
    \mathcal{I} \coloneqq \langle x_Sx_{S'} \; | \; S \text{ and } S' \text{ incomparable} \rangle 
    \end{equation}
    (in which ``incomparable" means that neither $S \subseteq S'$ nor $S' \subseteq S$) and $\mathcal{J}$ is the ideal of linear relations
    \begin{equation}
    \label{eq:linear}
        \mathcal{J} \coloneqq \left\langle \left.\sum_{S \ni e} x_S - \sum_{S \ni e'} x_S \; \right| \; \text{ distinct }e, e' \in E_i \text{ for some } i\right\rangle.
    \end{equation}
\end{proposition}

In fact, although the divisors $x_S$ for $S \in \decoseto$ are manifestly generators of the Chow ring $A^{\ast}(\Sigma^{\pi})$, another generating set has more elegant intersection-theoretic properties (and has a natural geometric interpretation explained in \cref{sec:psiclasses}). Namely, for each $S \in \decoseto$, we define
\begin{equation} \label{eq:hS}
h_S  \coloneqq  \sum_{S' \cap S \neq \emptyset} x_{S'}.
\end{equation}
These indeed generate $A^*(\Sigma^\pi)$, as  a result of the following lemma.

\begin{lemma}
    \label{lem:hStoxS}
For any $S \in \decoseto$, one has
\[
x_S = \sum_{\substack{U,T \in \decoseto\\ U\subseteq T \supseteq S}} (-1)^{|T|+|S|+|U|+1}h_U.\]
\end{lemma}
\begin{proof}
It is helpful to first express $h_S$ in terms of yet another generating set $f_S$, defined by
\begin{equation}
    \label{eq:fS}
f_T \coloneqq  \sum_{Z\supseteq T} x_Z
\end{equation}
for each $T \in \decoseto$.  Then
\begin{equation}
    \label{eq:fStohS}
h_S = \sum_{T\subseteq S} (-1)^{|T|+1} f_T,
\end{equation}
as one sees from the following standard inclusion-exclusion argument.  The right-hand side can be expanded as 
\begin{equation}
\label{eq:fTexpanded}\sum_{i \in S} f_i - \sum_{i,j \in S} f_{ij} + \sum_{i,j,k \in S} f_{ijk} - \cdots + (-1)^{|S|} f_S.
\end{equation}
The first term of these sums is equal to
\[\sum_{\substack{i \in S\\ T \ni i}} x_T,\]
and while each of these $x_T$'s appears in the definition of $h_S$, those for which $T$ contains two distinct elements of $S$ are double-counted.  The second summand of \eqref{eq:fTexpanded} subtracts these, but this double-counts those $x_T$ for which $T$ contains three distinct elements of $S$, and so on.

Since any interval in the poset $\decoseto$ is isomorphic to an interval in the Boolean lattice, the M\"obius function of this poset is $\mu(T,S)=(-1)^{|T|+|S|}$ for any $S\subseteq T$. Thus, the relation \eqref{eq:fStohS} can be inverted via M\"obius inversion (which is effectively inclusion-exclusion again) to yield
\begin{equation}
    \label{eq:hStofS}
f_S = \sum_{T \subseteq S} (-1)^{|T|+1} h_T,
\end{equation}
and by the same token, the relation \eqref{eq:fS} can be inverted to yield
\begin{equation}
     \label{eq:fStoxS}
x_S = \sum_{T \supseteq S} (-1)^{|T|+|S|}f_T.
\end{equation}
Combining these two equations gives
\[x_S  = \sum_{T \supseteq S} (-1)^{|T|+|S|} \left( \sum_{U \subseteq T} (-1)^{|U|+1} h_U\right),\]
which is precisely the statement of the lemma.
\end{proof}

\begin{remark}
It is occasionally convenient to allow that $S = \emptyset$, in which case we set $x_{\emptyset} = h_{\emptyset} = 0$.  However, we caution the reader that \cref{lem:hStoxS} does not hold for $S = \emptyset$.
\end{remark}

\section{Normal complexes}
\label{sec:normalcomplexes}

Recall that \cref{thm:B} is a statement about the degree of the top power of a particular divisor on $\Sigma^\pi$.  Such degrees have been related to the volumes of {\bf normal complexes} introduced by Nathanson--Ross \cite{NR23}. In this section, we review the relevant background---in a slightly less general setting than the framework of \cite{NR23}---and study its application to the $\pi$-colored fan.

\subsection{Background on normal complexes}

Let $N$ be a lattice and let $\Sigma$ be a fan in the vector space $N_\R \coloneqq  N \otimes \R$.  Assume that $\Sigma$ is
\begin{itemize}
    \item unimodular, meaning that every cone has generators that can be extended to a basis of $N$;
    \item pure of dimension $n$, meaning that all maximal cones are $n$-dimensional;
    \item tropical and balanced,\footnote{The definition of tropical fan is generally stated in terms of the existence of a weight function on the maximal cones of $\Sigma$ under which a weighted balancing condition is satisfied; see, for instance, \cite[Section 2.7]{NR23}.  The requirement that $\Sigma$ is balanced in our case means that the weight function is identically $1$, and it is a result of \cite[Proposition 5.6]{AHK} that this is equivalent to the existence of a degree map as stated.  It is worth noting that the definition of the balancing condition in \cite{AHK} and \cite{NR23} is subtlely different from the condition used in some other sources such as \cite[equation (3)]{FultonSturmfels}: the former is stated in terms of primitive integral generators $u_{\sigma \setminus \tau}$ of rays $\sigma \setminus \tau$, where $\sigma$ is a maximal cone and $\tau \subseteq \sigma$ a codimension-one face, whereas the latter replaces $u_{\sigma \setminus \tau}$ with a lattice point $n_{\sigma,\tau}$ whose image generates the quotient $N_\sigma/N_{\tau}$.  The two definitions coincide when $\Sigma$ is unimodular, which is sufficient for our purposes; avoiding this subtlety is the reason we assume $\Sigma$ is unimodular in this subsection.} meaning that there exists a linear degree map $\int_\Sigma\colon A^n(\Sigma) \rightarrow \R$ such that $\int_\Sigma X_\sigma = 1$ for each maximal cone $\sigma$ of $\Sigma$, where $X_\sigma$ is the product of the generators of $A^*(\Sigma)$ associated to the rays of $\sigma$.
\end{itemize}
Given such a fan, associated to any choice of inner product $\ast$ on $N_\R$ and any divisor $D \in A^1(\Sigma)_\R$, one can define a normal complex by truncating each cone of $\Sigma$ by affine hyperplanes normal to the rays; here, the notion of normal is determined by $\ast$ and the distance of the hyperplanes from the origin is determined by $D$.  More precisely, the definition is as follows.

For each ray $\rho$ of $\Sigma$, denote by $u_\rho \in N_\R$ the primitive integral generator (i.e., the first nonzero element of $N$ that lies on $\rho$), which exists because $\Sigma$ is unimodular.  Denoting by $\{x_\rho \; | \; \rho \in \Sigma(1)\}$ the generators of $A^*(\Sigma)$ associated to the rays of $\Sigma$, any divisor $D \in A^1(\Sigma)_\R$ can be expressed, not necessarily uniquely, as
\[D = \sum_{\rho \in \Sigma(1)} a_\rho x_\rho\]
for some $a_\rho \in \R$.  From here, associated to each cone $\sigma \in \Sigma$, one defines a polytope
\begin{equation} \label{eq:Psigma} P_{\sigma,\ast}(D) \coloneqq \{m \in \sigma \; | \;  m \ast u_\rho \leq a_\rho \ \text{ for all } \rho\in \sigma(1)\} \subseteq N_\R,
\end{equation}
where $\sigma(1)$ denotes the set of rays of $\sigma$.

The normal complex of $\Sigma$ associated to $\ast$ and $D$ is the union of these polytopes as $\sigma$ ranges over all maximal cones.  However, in order to ensure that these polytopes meet along faces---and therefore their union forms a polytopal complex---one must impose the following compatibility condition on $\ast$ and $D$.

\begin{definition}
\label{def:cubical}
    A divisor $D$ on $\Sigma$ is called {\bf pseudo-cubical} with respect to the inner product $\ast$ if the bounding hyperplanes of $P_{\sigma,\ast}(D)$ meet within $\sigma$ for all cones $\sigma$; that is, for each $\sigma \in \Sigma$ (not necessarily maximal), we have
    \[\sigma \cap \{m \in N_\R \; | \; m \ast u_\rho = a_\rho \text{ for all } \rho \in \sigma(1)\} \neq \emptyset.\]
    The divisor $D$ is {\bf cubical} with respect to $\ast$ if \[\sigma^\circ \cap \{m \in N_\R \; | \; m \ast u_\rho = a_\rho \text{ for all } \rho \in \sigma(1)\} \neq \emptyset\] for all $\sigma \in \Sigma$, where $\sigma^\circ$ denotes the interior of $\sigma$.
\end{definition}

From here, one can define the normal complex precisely as follows.

\begin{definition}
The {\bf normal complex} $\NC_{\Sigma,\ast}(D)$ of $\Sigma$ with respect to $\ast$ and $D$ is the union
\begin{equation}\label{eq:defNCD}
\NC_{\Sigma,\ast}(D) \coloneqq \bigcup_{\sigma\in \Sigma(n)} P_{\sigma,\ast}(D).
\end{equation}
It has the structure of a polytopal complex when $D$ is pseudo-cubical.
\end{definition} 

We refer the reader to \cref{sec:NCcolfan} for several examples of normal complexes in the specific context relevant to the current work.

In the case where $\Sigma$ is complete, the normal complex $\NC_{\Sigma,\ast}(D)$ is the classical normal polytope associated to $D$.  Moreover, a fundamental result of toric geometry (see, for example, \cite[Corollary, page 111]{FultonToric}) states that, when $D$ is nef, the volume of its normal polytope is equal to the degree $\int_\Sigma D^n$.  The main theorem of \cite{NR23} asserts that the analogous result is true when $\Sigma$ is not necessarily complete, with the normal complex now playing the role of the normal polytope.

To state the result precisely, care must be taken in how the volume is defined.  Specifically, for each $\sigma \in \Sigma(n)$, let 
\[N_\sigma \coloneqq N \cap \Span_\R(\sigma),\]
and let
\[M_\sigma \coloneqq N_\sigma^{\vee} = \Hom_\Z(N_\sigma, \Z).\]
These are lattices in different vector spaces, but the inner product $\ast$ allows one to view them both as lattices in the same space $\Span_\R(\sigma)$.  In this way, one can define the volume of any polytope in $\Span_\R(\sigma)$ by declaring
\[
\Vol_{\sigma,\ast}(\text{any $n$-simplex unimodular with respect to $M_\sigma$}) = 1,
\]
where {\bf unimodular} means that the simplex is lattice-equivalent to the $n$-simplex with vertices at $\mathbf{0}$ and the standard basis vectors.  This definition of volume, in particular, allows us to define the volume of the polytope $P_{\sigma,\ast}(D) \subseteq \Span_\R(\sigma)$, and adding these over each maximal cone defines the volume of the normal complex:
\begin{equation}
    \label{eq:volNCD}
    \vol_{\Sigma,\ast} \left(\NC_{\Sigma,\ast}(D)\right)  \coloneqq   \sum_{\sigma\in \Sigma(n)} \vol_{\sigma,\ast} \left(P_{\sigma,\ast}(D)\right).
\end{equation} 
The main theorem of \cite{NR23}, in the generality we will need, is the following.

\begin{thm}\label{thm:NR23main} \cite[Theorem 6.3]{NR23} Let $\Sigma$ be a unimodular, pure $n$-dimensional, balanced tropical fan in $N_\R$, let $\ast$ be an inner product on $N_\R$, and let $D \in A^1(\Sigma)_\R$ be a divisor that is pseudo-cubical with respect to $\ast$.  Then 
\[
\int_\Sigma D^n = \vol_{\Sigma,\ast} \left(\NC_{\Sigma,\ast}(D)\right).
\]
\end{thm}

\subsection{Application to the \texorpdfstring{$\pi$}{pi}-colored fan}  \label{sec:NCcolfan} We now specialize the previous subsection to the case in which $\Sigma$ is the $\pi$-colored fan $\Sigma^\pi$.  In this case, the lattice $N$ is
\[ N^\pi \coloneqq  \dfrac{\Z^{E_1}}{\Z \e_{E_1}} \times \dots \times \dfrac{\Z^{E_n}}{\Z \e_{E_n}},\]
so that $N^\pi_\R$ is as in \eqref{eq:NpiR}.  It is straightforward to see that $\Sigma^\pi$ is indeed unimodular and pure $n$-dimensional.  To see that it is a balanced tropical fan, one must verify the condition of \cite[equation (2.16)]{NR23} with $w(\sigma)=1$: namely, for each chain $\scC$ of length $n-1$,
\begin{equation}
    \label{eq:balancing}
\sum_{\substack{\scC' \in \maxchain\\ \sigma_{\scC} \subseteq \sigma_{\scC'}}} \ove_{\scC' \setminus \scC} \in \Span_{\R}(\sigma_{\scC}),
\end{equation}
where $\maxchain$ denotes the set of maximal chains and $\scC' \setminus \scC$ is the unique colored set in the chain $\scC'$ that is not in the chain $\scC$.  If the maximal element of $\scC$ has size $n$, then we can write $\scC = (S_1 \subsetneq \cdots \subsetneq \widehat{S_i} \subsetneq \cdots \subsetneq S_n)$ for some $i \in [n]$, with $|S_j| = j$ for all $j \in [n] \setminus \{i\}$.  In this case, one sees that the sum in \eqref{eq:balancing} equals
\[\sum_{x \in S_{i+1} \setminus S_{i-1}} \ove_{S_{i-1} \cup\{x\}} = \ove_{S_{i-1}}  + \ove_{S_{i+1}},\]
which indeed lies in $\Span_{\R}(\sigma_{\scC})$.  If the maximal element of $\scC$ does not have size $n$, then the sum \eqref{eq:balancing} is $\sum_{j \in E_i} \ove_j$ for some $i \in [n]$, which equals zero and therefore also lies in $\Span_{\R}(\sigma_{\scC})$.

To apply the machinery of \cite{NR23}, we must now choose an inner product on 
 $N^\pi_\R$.  To define the inner product, recall that $\overline{\R}^{E_i}$ denotes the image of $\R^{E_i}$ in $N^\pi_\R$.  
Choose an inner product $\ast_i$ on each $\overline{\R}^{E_i}$ with \[\ove_j \ast_{i} \ove_j = 1,\]  for all $j \in E_i$, and set $\ast \coloneqq  \ast_1 \times \cdots \times \ast_n$.

\begin{remark}
    There is a non-canonical isomorphism $\overline{\R}^{E_i} \cong \R^{|E_i|-1}$ given by choosing $a \in E_i$ and sending $\{\ove_j \; | \; j \neq a\}$ to the standard basis vectors, while sending $\ove_a$ to the vector $(-1, -1, \ldots, -1)$.  We note that $\ast_i$ is not the standard inner product on $\R^{|E_i|-1}$ under this isomorphism unless $|E_i|=2$.  For example, if $E_i = \{1,2,3\}$ and we choose the isomorphism $\overline{\R}^{E_i} \cong \R^{2}$ given by
    \[\ove_1 \mapsto (-1,-1), \quad \ove_2 \mapsto (1,0), \quad \ove_3 \mapsto (0,1),\]
    then $\ast_i$ is not the standard inner product on $\R^2$ but can instead be taken to be
    \[(x_1,y_1)\ast_i(x_2,y_2) \coloneqq x_1x_2+y_1y_2-\frac{1}{2}(x_1y_2+x_2y_1). \]
    This example can be generalized to all dimensions to give an explicit formula for $\ast_i$.
\end{remark}

Under this choice of inner product, we can describe the normal complex $\NC_{\Sigma^{\pi}, \ast}(D)$ explicitly as follows.  First, note that by \cref{def:Sigma}, the maximal cones of $\Sigma^\pi$ are of the form $\sigma_\scC$ in which $\scC$ is a maximal chain in $\decoset$, so we can rewrite \eqref{eq:defNCD} as 
\[\NC_{\Sigma^{\pi}, \ast}(D) = \bigcup_{\scC \in \maxchain} P_{\sigma_{\scC}, \ast} (D).\]
Grouping these chains according to their maximal element, which is necessarily an element of $\decosetmax$, we have
    \begin{equation}
    \label{eq:NCSigmapi}
    \NC_{\Sigma^{\pi}, \ast}(D) = \bigcup_{T \in \decosetmax} \bigcup_{\scC \in \mchain(T)} P_{\sigma_{\scC, \ast}} (D),
    \end{equation}
where $\mchain(T)$ is the subset of $\maxchain$ consisting of maximal chains with $T$ as their maximal element.  Note, in this grouping, that all of polytopes $P_{\sigma_{\scC, \ast}} (D)$ lie in the same subspace $\overline{\mathbb{R}}^T \subseteq N^\pi_\R$.  Moreover, the volume functions $\vol_{\sigma_{\scC}, \ast}$ for any $\scC \in \mchain(T)$ are all restrictions of the same volume function on $\overline{\R}^T$, which we now describe.

To do so, note that for any $T \in \decosetmax$, the fact that $T$ is colored implies that $\overline{\R}^T \cong \R^T$, and the fact that it is maximal further implies that $\overline{\R}^T$ is isomorphic to $\R^n$ via a basis of the form
\begin{equation}
    \label{eq:ONbasis}
\{\ove_i \; 
| \; i \in T\}.
\end{equation} 
Now, let $\vol_{T}$ be the volume function on $\R^n \cong \overline{\R}^T$ with
\begin{equation} \label{eq:VolT} \vol_T(\text{standard $n$-simplex}) \coloneqq 1,\end{equation}
where the standard $n$-simplex refers to the convex hull of $\mathbf{0}$ and the standard basis vectors in $\mathbb{R}^n$.  Then we have the following lemma. 

\begin{lemma}
\label{lem:volT}
For any maximal cone $\sigma_{\scC}$ of $\Sigma^\pi$ associated to a chain $\scC \in \mchain(T)$, the volume function $\vol_{\sigma_{\scC}, \ast}$ on $\Span_\R(\sigma_{\scC}) \subseteq \overline{\R}^T$ is the restriction of $\vol_T$.
\end{lemma}
\begin{proof}
Fix a maximal cone $\sigma_{\scC}$ as in the statement of the lemma.  Then the set \eqref{eq:ONbasis} is both an orthonormal basis of $\overline{\R}^T$ and a $\Z$-basis of $N^\pi \cap \overline{\R}^T$.  It follows that the isomorphism $N^\pi_\R \cong (N^\pi_\R)^{\vee}$ given by $\ast$ identifies the lattice $N^\pi_{\sigma_{\scC}}$ with the lattice $M^\pi_{\sigma_{\scC}}$.  Thus, under the isomorphism $\overline{\R}^T \cong \R^n$ provided by this basis, the standard $n$-simplex is unimodular with respect to $M^\pi_{\sigma_{\scC}} = N^\pi_{\sigma_{\scC}}$, and the lemma follows.
\end{proof}

Combining \cref{lem:volT} with the definition of volume in \eqref{eq:volNCD}, one sees that for any divisor $D$ on $\Sigma^\pi$, the volume of the normal complex $\NC_{\Sigma^\pi, \ast}(D)$ is given by
\begin{equation}
\label{eq:volpi}
\vol_{\pi} (C_{\Sigma^\pi, \ast}(D)) \coloneqq \sum_{T \in \decosetmax} \vol_T\left( \bigcup_{\scC \in \mchain(T)} P_{\sigma_\scC,\ast}(D)\right).
\end{equation}

Let us illustrate this normal complex and its volume in some examples.  We specifically consider cases where the divisor $D$ is $\sum_{S \in \decoset} h_S$, as this will play a key role in the proof of \cref{thm:h-integrals} below. 

\begin{example}
\label{ex:cubical}
    As in \cref{ex:SigmaB2}, let $E = \{{1},\bar{1}\} \sqcup \{{2},\bar{2}\}$ and consider the divisor $D = \sum_{S \in \decoset} h_S$.     A straightforward computation from the definition \eqref{eq:hS} of $h_S$ shows that $D$ can be expanded as
    \begin{equation}
        \label{eq:SigmaB2cubical}
        D = 3\left(x_{\{{1}\}} + x_{\{\bar{1}\}} + x_{\{{2}\}} +  x_{\{\bar{2}\}}\right) + 5 \left( x_{\{{1},{2}\}} + x_{\{\bar{1},{2}\}} + x_{\{{1},\bar{2}\}} + x_{\{\bar{1},\bar{2}\}}\right).
    \end{equation}
    The normal complex $\NC_{\Sigma^\pi, \ast}(D)$ is depicted in the leftmost part of Figure~\ref{fig:SigmaB2cubical}.  Note that it is bounded by hyperplanes normal to the eight rays, and that the normal hyperplanes to the rays of any maximal cone $\sigma_{\scC}$ meet in the interior of $\sigma_{\scC}$.  This shows that $D$ is cubical, and it illustrates the reason for the terminology: the cubical condition ensures that $P_{\sigma_{\scC}, \ast}(D)$ is combinatorially a cube.

    There are four choices of $T \in \decosetmax$ in this example, and the corresponding subspaces $\overline{\R}^T \subseteq N^\pi_\R$ are the four quadrants in Figure~\ref{fig:SigmaB2cubical}.  Each quadrant contains two polytopes $P_{\sigma_\scC, \ast}(D)$, corresponding to the two choices of $\scC \in \mchain(T)$.  The decomposition~\eqref{eq:volpi} then says that  $\vol_\pi\left(\NC_{\Sigma^\pi, \ast}(D)\right)$ is computed by assigning volume $1$ to the standard $n$-simplex within each quadrant. From the rightmost part of Figure~\ref{fig:SigmaB2cubical} we deduce that the volume of $\NC_{\Sigma^\pi, \ast}(D)$ is 68.
        \begin{figure}[h!]
        \centering
\begin{tikzpicture}[xscale=0.55, yscale=0.55]
\draw [ao(english),fill=ao(english),opacity=0.3] (2,3) -- (3,2) -- (3, -2) -- (2,-3) -- (-2,-3) -- (-3, -2) -- (-3, 2) -- (-2, 3)  -- (2,3); 

\draw [ao(english)]  (2,3) -- (3,2) -- (3, -2) -- (2,-3) -- (-2,-3) -- (-3, -2) -- (-3, 2) -- (-2, 3)  -- (2,3); 

\draw [opacity=0.2, ->] (0,0) --(3.6,0);
\draw [opacity=0.2, ->] (0,0) --(-3.6,0);
\draw [opacity=0.2, ->] (0,0) --(0,3.6);
\draw [opacity=0.2, ->] (0,0) --(0,-3.6);
\draw [opacity=0.2, ->] (0,0) --(3.3,3.3);
\draw [opacity=0.2, ->] (0,0) --(-3.3,3.3);
\draw [opacity=0.2, ->] (0,0) --(3.3,-3.3);
\draw [opacity=0.2, ->] (0,0) --(-3.3,-3.3);

\draw [->] (0,0) --(1,0);
\draw [->] (0,0) --(-1,0);
\draw [->] (0,0) --(0,1);
\draw [->] (0,0) --(0,-1);
\draw [->] (0,0) --(1,1)  node[above right] {$\ove_{\{\blu{1},\red{2}\}}$};
\draw [->] (0,0) --(-1,1)  node[above left] {$\ove_{\{\blub{1},\red{2}\}}$};
\draw [->] (0,0) --(1,-1) node[below right] {$\ove_{\{\blu{1},\redb{2}\}}$};
\draw [->] (0,0) --(-1,-1)  node[below left] {$\ove_{\{\blub{1},\redb{2}\}}$};

\node at (0, -1.8) {$\ove_{\redb{2}}$};
\node at (0, 1.8) {$\ove_{\red{2}}$};
\node at (1.8,0) {$\ove_{\blu{1}}$};
\node at (-1.8,0) {$\ove_{\blub{1}}$};

\begin{scope}[shift={(8.5,0)}]

\draw [ao(english),fill=ao(english),opacity=0.1] (3,0) -- (3, -2) -- (2,-3) -- (-2,-3) -- (-3, -2) -- (-3, 2) -- (-2, 3)  -- (2,3) -- (2.5,2.5) -- (0,0);

\filldraw [ao(english), opacity=0.3] (0,0) -- (3,0) -- (3,2) -- (2.5, 2.5) -- (0,0);

\draw [opacity=0.4, ->] (0,0) --(3.6,0);
\draw [ opacity=0.4, ->] (0,0) --(3.3,3.3);

\draw [blue, thick] (5-2.7,2.7) -- (3.2, 5-3.2);
\draw [magenta, thick] (3,-0.3) -- (3, 2.3);

\node [magenta] at (2,-0.4) {\small $x=3$};
\node [blue] at (0.8, 2.6) {\small $x+y=5$};
\filldraw [ao(english)] (3,2) circle[radius=0.6mm] node [right] {\small $(3,2)$};

\end{scope}
\begin{scope}[shift={(17,0)}]

\draw [ao(english),fill=ao(english),opacity=0.1] (3,2) -- (3, -2) -- (2,-3) -- (-2,-3) -- (-3, -2) -- (-3, 2) -- (-2, 3)  -- (2,3) -- (3,2); 

\draw [opacity=0.2] (-3,2) -- (3,2);
\draw [opacity=0.2] (-3,-2) -- (3,-2);
\draw [opacity=0.2] (-3,1) -- (3,1);
\draw [opacity=0.2] (-3,-1) -- (3,-1);
\draw [opacity=0.2] (-3,0) -- (3,0);

\draw [opacity=0.2] (2,-3) -- (2,3);
\draw [opacity=0.2] (-2,-3) -- (-2,3);
\draw [opacity=0.2] (1,-3) -- (1,3);
\draw [opacity=0.2] (-1,-3) -- (-1,3);
\draw [opacity=0.2] (0,3) -- (0,-3);

\draw [opacity=0.2] (1,3) -- (3,1);
\draw [opacity=0.2] (0,3) -- (3,0) -- (0,-3) -- (-3,0) -- (0,3);
\draw [opacity=0.2] (0,2) -- (2,0) -- (0,-2) -- (-2,0) -- (0,2);
\draw [opacity=0.2] (0,1) -- (1,0) -- (0,-1) -- (-1,0) -- (0,1);
\draw [opacity=0.2] (-1,3) -- (-3,1);
\draw [opacity=0.2] (-3,-1) -- (-1,-3);
\draw [opacity=0.2] (3,-1) -- (1,-3);

\draw [opacity=0.2] (2,3) -- (3,2) -- (3, -2) -- (2,-3) -- (-2,-3) -- (-3, -2) -- (-3, 2) -- (-2, 3)  -- (2,3); 
\end{scope}
\end{tikzpicture}
        \caption{On the left, the normal complex $\NC_{\Sigma^\pi, \ast}(D)$ from \cref{ex:cubical}.  In the middle, the polytope $P_{\sigma_{\scC},\ast}(D)$ and its bounding hyperplanes, where $\sigma_\scC$ is the maximal cone associated with the chain $\scC=\left(\{{1}\} \subseteq \{{1},{2}\}\right)$. On the right, $\NC_{\Sigma^\pi, \ast}(D)$ is subdivided into simplices, each of volume 1.}
        \label{fig:SigmaB2cubical}
    \end{figure}
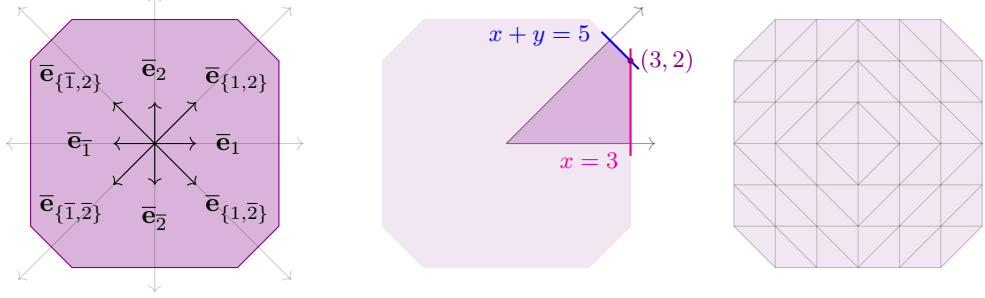
\end{example}

\begin{example} \label{ex:sumHnotcubical}
By contrast to the previous example, the divisor $D = \sum_{S \in \decoset} h_S$ is not necessarily cubical if the partition is not uniform.  For instance, as in \cref{ex:nonuniform}, let $E$ and its partition $\pi$ be defined by $E = \{a,b,c\} \sqcup \{1,2\}$.  Then the divisor $D =\sum_{S \in \decoset} h_S$ can be expanded as
\[
3 (x_{\{a\}} + x_{\{b\}} + x_{\{c\}}) + 4 (x_{\{1\}} + x_{\{2\}}) + 6 (x_{\{1,a\}}+ x_{\{1,b\}}+ x_{\{1,c\}} + x_{\{2,a\}} + x_{\{2,b\}} + x_{\{2,c\}}).
\]
This divisor is pseudo-cubical but not cubical: the normal hyperplanes to the rays of the cone $\sigma_{\scC}$ on the right-hand part of Figure~\ref{fig:pseudocubical} meet on the boundary of $\sigma_{\scC}$.

 \begin{figure}[h!]
    \centering
    \tdplotsetmaincoords{30}{35} 
\begin{tikzpicture}[tdplot_main_coords,scale=0.65]
 
    \draw[  ->] (0,0,0) -- (0,1,1);
	\draw[  ->] (0,0,0) -- (0,1,-1);
    \draw[  ->] (0,0,0) -- (0,1,0) node[above right] {$\ove_{\red{b}}$};
	\draw[  ->] (0,0,0) -- (0,0,1) node[above] {$\ove_{\blu{1}}$};
    \draw[  ->] (0,0,0) -- (1,0,-1);
	\draw[  ->] (0,0,0) -- (1,0,1);
    \draw[  ->] (0,0,0) -- (1,0,0) node[right] {$\ove_{\red{a}}$};
    \draw[  ->] (0,0,0) -- (0,0,-1) node[below] {$\ove_{\blu{2}}$};
    \draw[  ->] (0,0,0) -- (-1,-1,0) node[left] {$\ove_{\red{c}}$};
    \draw[  ->] (0,0,0) -- (-1,-1,1);
    \draw[  ->] (0,0,0) -- (-1,-1,-1);

    \draw[->,opacity=0.2] (0,0,0) -- (0,3.2,3.2);
	\draw[->,opacity=0.2] (0,0,0) -- (0,3.2,-3.2);
    \draw[->,opacity=0.2] (0,0,0) -- (0,3.2,0);
	\draw[->,opacity=0.2] (0,0,0) -- (0,0,4.2);
    \draw[->,opacity=0.2] (0,0,0) -- (3.2,0,-3.2);
	\draw[->,opacity=0.2] (0,0,0) -- (3.2,0,3.2);
    \draw[->,opacity=0.2] (0,0,0) -- (3.2,0,0);
    \draw[->,opacity=0.2] (0,0,0) -- (0,0,-4.2);
    \draw[->,opacity=0.2] (0,0,0) -- (-3.2,-3.2,0);
    \draw[->,opacity=0.2] (0,0,0) -- (-3.2,-3.2,3.2);
    \draw[->,opacity=0.2] (0,0,0) -- (-3.2,-3.2,-3.2);
    
    \draw [ao(english), opacity=0.3, fill=ao(english)](0,0,4) -- (2,0,4) -- (3,0,3) -- (3,0,-3) -- (2,0,-4) -- (0,0,-4) -- (0,0,4);
    \draw [ao(english), opacity=0.3, fill=ao(english)](0,0,4) -- (0,2,4) -- (0,3,3) -- (0,3,-3) -- (0,2,-4) -- (0,0,-4);
    \draw [ao(english), opacity=0.3, fill=ao(english)](0,0,4) -- (-2,-2,4) -- (-3,-3,3) -- (-3,-3,-3) -- (-2,-2,-4) -- (0,0,-4);

   \begin{scope}[shift={(7,4.9,0)}]
           
    \draw [ao(english), opacity=0.1, fill=ao(english)](0,0,4) -- (2,0,4) -- (3,0,3) -- (3,0,-3) -- (2,0,-4) -- (0,0,-4) -- (0,0,4);
    \draw [ao(english), opacity=0.1, fill=ao(english)](0,0,4) -- (0,2,4) -- (0,3,3) -- (0,3,-3) -- (0,2,-4) -- (0,0,-4);
    \draw [ao(english), opacity=0.1, fill=ao(english)](0,0,4) -- (-2,-2,4) -- (-3,-3,3) -- (-3,-3,-3) -- (-2,-2,-4) -- (0,0,-4);

    \draw [ao(english), opacity=0.3, fill=ao(english)](0,0,0) -- (3,0,3) -- (3,0,0) -- (0,0,0);

 	\draw[->,opacity=0.4] (0,0,0) -- (3.7,0,3.7);
	\draw[->,opacity=0.4] (0,0,0) -- (3.7,0,0);
    \draw [magenta] (3,0, 3.7) -- (3,0,-0.2);
    \draw [blue] (3+0.3,0,3-0.3) -- (3-0.6,0,3+0.6);
    \filldraw [ao(english)] (3,0,3) circle[radius=0.6mm];

      \end{scope}
    
  \end{tikzpicture}
  \caption{On the left, the normal complex $\NC_{\Sigma^\pi, \ast}(D)$ from \cref{ex:sumHnotcubical}. On the right, the polytope $P_{\sigma_{\scC},\ast}(D)$ and its bounding hyperplanes, where $\sigma_\scC$ is the maximal cone associated with the chain $\scC=\left(\{a\} \subseteq \{1,a\}\right)$. Note that the intersection of the bounding hyperplanes lies on the boundary of $\sigma_{\scC}$.}
    \label{fig:pseudocubical}
\end{figure}
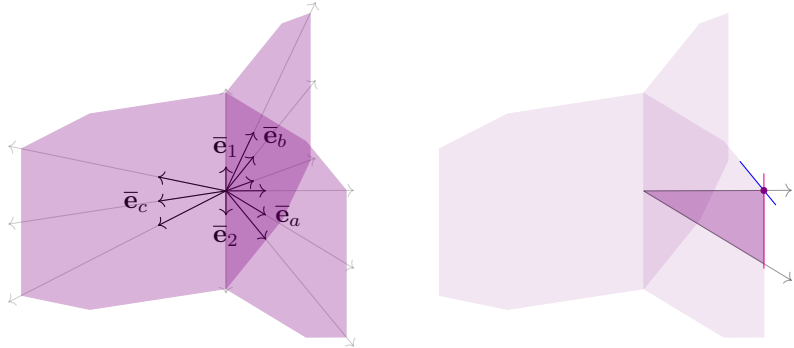
\end{example}

For later reference, we describe the pseudo-cubical condition for the $\pi$-colored fan explicitly.  To do so, we express a divisor $D \in A^1(\Sigma^\pi)_\R$ as a linear combination of the generators $x_S$ for $S \in \decoseto$, in which case we have the following lemma.

\begin{lemma}
\label{lem:DMpseudocubical} Let
\[D = \sum_{S \in \decoseto} \coeff(S) x_S \in A^1(\Sigma^\pi)_\R\]
be such that $\coeff(S)\geq 0$ for all $S$.  Then $D$ is pseudo-cubical if and only if, for any maximal chain $\scC=\big(S_1 \subsetneq \dots \subsetneq S_n \big)$ of nonempty colored sets and for every $i \in [n-1]$ and $j \in [n]$, the following conditions hold:
\begin{equation} \label{eq:coeff}
2\coeff(S_i) \geq \coeff(S_{i-1}) + \coeff(S_{i+1}) \qquad \text{and} \qquad \coeff(S_j) \geq \coeff(S_{j-1}),
\end{equation} where $c(S_0)=0$.  Furthermore, $D$ is cubical if and only if the inequalities are all strict. 
\end{lemma}
\begin{proof} We first show that, for the $\pi$-colored fan, it suffices to check the cubicality condition of \cref{def:cubical} on maximal cones.  This follows from the observation that, for any maximal cone $\sigma$ of $\Sigma^\pi$ and any face $\tau \subseteq \sigma$, the orthogonal projection of $\Span_{\R}(\sigma)$ onto $\Span_{\R}(\tau)$ takes $\sigma$ to $\tau$ and $\sigma^{\circ}$ to $\tau^{\circ}$.  This is tedious but straightforward to check: first, by relabelling elements of $E$, we can assume that $\sigma$ is the cone associated with the maximal chain $\scC=( \{1\} \subsetneq \{1,2\} \subsetneq \dots \subsetneq [n])$ and that $\tau$ is the cone associated to a chain 
$(S_1 \subsetneq \cdots \subsetneq S_k)$ which is refined by $\scC$. 
For every $i \in [k]$, we set $T_i \coloneqq S_i \setminus S_{i-1}$, which necessarily consists of consecutive integers.  Then, since
$\{\ove_{T_1}, \ldots \ove_{T_k}\}$ is an orthogonal basis of $\Span_{\R}(\tau)$, the orthogonal projection of $\Span_{\R}(\sigma)$ onto $\Span_{\R}(\tau)$ sends
\[\vec{x} = c_1 \ove_1 + c_2 \ove_{\{1,2\}} + \cdots + c_n \ove_{[n]} \in \sigma \]
to 
\[ \sum_{i=1}^k \frac{\vec{x} \ast \ove_{T_i}}{|T_i|} \ove_{T_i} = \sum_{i=1}^k \left(  \frac{1}{|T_i|} \sum_{j \in T_i} (c_j + c_{j+1} + \cdots + c_n) - \frac{1}{|T_{i+1}|}\sum_{j \in T_{i+1}} (c_j + \cdots + c_n) \right) \ove_{S_i}.\]
Writing $T_i= \{\ell+1, \ell+2, \ldots, \ell+a\}$ and $T_{i+1} = \{j+1, j+2, \ldots, j + b\}$ for $\ell+a  \leq j$, the coefficient of $\ove_{S_i}$ in the above summation is
\[
\sum_{x=1}^a \dfrac{x}{a}c_{\ell + x} + \sum_{y=\ell+a+1}^{j} c_y +  \sum_{z=1}^{b} \left(1-\dfrac{z}{b} \right) c_{j+z}\]
This is manifestly non-negative whenever $c_j \geq 0$ for all $j$ and positive whenever $c_j > 0$ for all $j$.  Therefore, the orthogonal projection indeed sends $\sigma$ to $\tau$ and $\sigma^{\circ}$ to $\tau^{\circ}$.

Thus, to check that $D$ is (pseudo-)cubical, one must only check the condition of \cref{def:cubical} on each maximal cone of $\Sigma^\pi$.  These are of the form $\sigma_{\scC}$ for a maximal chain $\scC = (S_1 \subsetneq \cdots \subsetneq S_n)$, and the pseudo-cubical condition is that
\begin{equation}\label{eq:cubicalsetproof}
 \sigma_{\scC} \cap \left\lbrace \left. \vec{x} = \sum_{i \in S_n} x_i \ove_i \in \overline{\R}^{S_n} \; \right\vert \; \vec{x} \ast \ove_S = \coeff(S) \; \text{ for all } \; S \in \scC  \right\rbrace \ne \varnothing,
\end{equation}
while the cubical condition is that these intersections all lie within $\sigma_{
\scC}^{\circ}$.  Since the chain $\scC$ is maximal, for every $i\neq j \in S_n$ we have $\ove_i \ast \ove_j = \delta_{ij}$.  Thus, the $n$ conditions $\vec{x} \ast \ove_S = \coeff(S)$ for $S \in \scC$  yield the equations
   \[ \sum_{i \in S_k} x_{i}  = c(S_k) \qquad \text{ for all } k \in [n] \\
    \]
on the coordinates of $\vec{x}$.  These equations determine $\vec{x}$, meaning that \eqref{eq:cubicalsetproof} consists of a single point.  Namely, if we order the elements of $S_n = \{j_1, \ldots, j_n\}$ by the condition that $j_i \in T_i$ for each $i$, then we have
\[\vec{x} = \sum_{i=1}^n \left(c(S_i) - c(S_{i-1})\right)\ove_{j_i}.\]
This $\vec{x}$ belongs to 
\[\sigma_{\scC} = \cone\left(\ove_{j_1}, \;\ove_{j_1} + \ove_{j_2}, \; \ove_{j_1} + \ove_{j_2} + \ove_{j_3}, \ldots,\; \ove_{j_1} + \cdots + \ove_{j_n}\right)\]
if and only if  $x_{j_{i+1}} \leq x_{j_i}$ for every $i \in [n-1]$ and $x_{j} \geq 0$ for every $j \in[n]$, and it belongs to $\sigma_{\scC}^{\circ}$ if and only if these inequalities are all strict.  This completes the proof. 
\end{proof}

\begin{example}
    One can use \cref{lem:DMpseudocubical} to verify that the divisor $D$ from \cref{ex:cubical} is cubical.  For instance, for the chain $\scC = \left(\{{1}\} \subseteq \{{1},{2}\}\right)$, the relevant coefficients of $D$ satisfy
    \[2 \cdot c(\{{1}\}) = 2 \cdot 3 > 0+5 = c(\emptyset) + c(\{{1},{2}\}) \quad \text{ and } \quad c(\{1, 2\}) = 5 > 3 = c(\{1\}).\]
    On the other hand, one can also use \cref{lem:DMpseudocubical} to verify that the divisor $D$ from \cref{ex:sumHnotcubical} is pseudo-cubical but not cubical.  For instance, the chain  $\scC = \left( \{a\} \subseteq \{1,a\}\right)$ gives the equality $2 \cdot 3 = 0+6$. 
\end{example}

Equipped with the explicit descriptions of the pseudo-cubical condition and the volume function furnished by the previous two lemmas, we can restate the results of \cite{NR23} for the $\pi$-colored fan as follows.

\begin{corollary}\label{cor:NR23-application-Sr}
Let $D = \sum_{S \in \decoseto} c(S) x_S \in A^1(\Sigma^\pi)_\R$ be such that $c(S) \geq 0$ for all $S$.  Assume that $D$ is pseudo-cubical with respect to $\ast$, or equivalently, that the coefficients $c(S)$ satisfy the inequalities \eqref{eq:coeff} for all maximal chains $\scC$ of colored sets.  Then
\[
\int_{\Sigma^\pi} D^n = \vol_\pi\left( \NC_{\Sigma^\pi, \ast}(D)\right)
\]
in which $\vol_\pi$ is defined by \eqref{eq:volpi}.
\end{corollary}

In order to use \cref{cor:NR23-application-Sr} to prove \cref{thm:multimatroid_integral}, we must now specialize to the case $D = D_\bfM$ for a multimatroid $\bfM$.  For this, we turn in the next section to giving the requisite preliminary definitions on multimatroids.

\section{Multimatroids}
\label{sec:multimatroids}
We begin this section by introducing the concept of an $\R$-multimatroid, which can be seen as a generalization of the multimatroids introduced in \cite[Section 3]{bouchet1997multimatroids}. As explained in the introduction, the motivation for this extension is that it allows us to extend the equality of \cref{cor:NR23-application-Sr} to settings where the divisor $D= D_\bfM$ is not necessarily cubical.  Throughout, we assume that the data of $(E,\pi)$ is fixed.

\subsection{Definition of \texorpdfstring{$\R$}{R}-multimatroids}

First, we recall the definition of multimatroid from \cite{bouchet1997multimatroids}.

\begin{definition}
\label{def:multimatroid}
A \textbf{multimatroid} $\mathbf{M}$ on $(E, \pi)$ is a rank function $\rkk \colon \decoset \to \N$ satisfying the following conditions:
\begin{enumerate}[label=(BR\arabic*)]
    \item \label{it:BR1} $\rkk(\varnothing) = 0$;
    \item \label{it:BR2} (\textit{monotonicity and boundedness}) for any $S \in \decoset$ and any $x \in E_i$ such that $E_i \cap S = \varnothing$, \[\rkk(S) \le \rkk(S \cup \{x\}) \le \rkk(S) + 1;\]  
    \item \label{it:BR3} (\textit{submodularity}) for any $S, T \in \decoset$ with $S \cup T \in \decoset$, \[\rkk(S \cup T) + \rkk(S \cap T) \le \rkk(S) + \rkk(T);\] 
    \item \label{it:BR4} for any $S \in \decoset$ and any pair $\{x, y\} \subseteq E_i$ such that $E_i \cap S = \varnothing$, either $\rkk(S \cup \{x\}) - \rkk(S) = 1$ or $\rkk(S \cup \{y\}) - \rkk(S) = 1$. 
\end{enumerate}
If $(E, \pi)$ is uniform with $|E_i| = r$ for each $i$, then $\bfM$ is referred to as an {\bf $r$-matroid.}
\end{definition}

One can always define a multimatroid on $(E,\pi)$ by setting $\rkk(S) = |S|$ for all $S \in \decoset$; this is referred to as the {\bf Boolean multimatroid}.  For a somewhat more motivated class of examples---which, in fact, was one of the original reasons for the definition of multimatroids---one can consider the collection of vertex splitters of the medial graph of an embedded (or, more generally, $4$-regular) graph; this is the main subject of \cite{ellis2013graphs}.

The generalization of the concept of multimatroid that we require allows the rank function to be $\R$-valued and removes the boundedness condition from \ref{it:BR2} for the two reasons stated at the end of the introduction, so our notion might be called a weak $\R$-multimatroid in the language of \cite{BakerBowler}.  We also remove condition \ref{it:BR4}, since that condition is not needed for any of our results,\footnote{We refer to \cite[Remark, page 633]{bouchet1997multimatroids} for a comment about \ref{it:BR4}. 
 It is worth noting that, although Bouchet originally viewed the structure defined by \ref{it:BR1}--\ref{it:BR3} as ``too weak to be interesting,'' our results show that one can indeed prove interesting results without imposing condition \ref{it:BR4}.} so our notion might be called a weak poly-$\R$-multimatroid in the language of \cite{Edmonds}.  To avoid this proliferation of qualifiers, we simply refer to our concept as an $\R$-multimatroid, and we define it precisely as follows.

\begin{definition}
\label{def:multimatroid-rank}
    An \textbf{$\R$-multimatroid} $\bfM$ on $(E,\pi)$ is a {rank function} $\rkk \colon\decoset \to \R$ satisfying
    \begin{enumerate}[label=(R\arabic*)]
        \item \label{it:R1} $\rkk(\varnothing) = 0$;
        \item \label{it:R2} (\textit{monotonicity}) for any $S, T \in\decoset$ with $S \subseteq T$, one has $\rkk(S) \le \rkk(T)$;
        \item \label{it:R3} (\textit{submodularity}) for any $S, T \in \decoset$ with $S \cup T \in \decoset$, \[\rkk(S \cup T) + \rkk(S \cap T) \leq  \rkk(S) + \rkk(T).\]
    \end{enumerate}
   \end{definition}

\begin{example}\label{ex:M12bar12}
Let $E=\{{1}, \bar{1}\} \sqcup \{ {2}, \bar{2}\}$.  Then one can define an $\R$-multimatroid on $E$ by
\begin{align*}
    &\rkk(\emptyset)=0;\\
    &\rkk(\{{1}\})=\rkk(\{\bar{1}\})=\rkk(\{{2}\})=\rkk(\{\bar{2}\})=\rkk(\{\bar{1},{2}\})=\rkk(\{{1},\bar{2}\})= 1;\\
    &\rkk(\{{1},{2}\})=\rkk(\{\bar{1},\bar{2}\})=2.
\end{align*}
In fact, this example is a multimatroid, as one can check from \cref{def:multimatroid}.
\end{example}

\begin{example} 
\label{ex:Rmultimatroid}
    Let $E=\{1, \bar{1}\} \sqcup \{2, \bar{2}\}$.  Then one can define an $\R$-multimatroid on $E$ by
    \begin{align*}
        &\rkk(\emptyset) = 0;\\
        &\rkk(1) = \rkk(\bar{1}) = 5;\\
        &\rkk(2) = \rkk(\bar{2}) = 4;\\
        &\rkk(\{1,2\}) = \rkk(\{\bar{1}, 2\}) = \rkk(\{1, \bar{2}\}) = \rkk(\{\bar{1}, \bar{2}\}) = 6.
    \end{align*}
    This is not a multimatroid because it does not satisfy the boundedness condition in \ref{it:BR2}.
\end{example}

\begin{remark} \label{rem:Rmultimatroid}
One advantage of the additional constraints in \cref{def:multimatroid} is that they allow one to alternatively describe multimatroids via their associated collection of independent sets, bases, or circuits, in the same way that matroids are often described. 
 It would be interesting to determine whether $\R$-multimatroids can analogously be described via their independent sets, bases, or circuits, but we currently do not know of such a description.
\end{remark}

By identifying each $\R$-multimatroid with its rank function, one can define a topological space of $\R$-multimatroids as follows.

\begin{definition} \label{sec:scM} The \textbf{space of $\R$-multimatroids on $(E,\pi)$}, denoted $\scM=\scM_{(E,\pi)}$, is the subset of the set of functions $\decoset \rightarrow \R$ satisfying the conditions of \cref{def:multimatroid-rank}.  Identifying the function space with $\R^{\decoset}$ and giving it the usual Euclidean topology, we have an embedding
\[\scM \subseteq \R^{\decoseto}\]
in which $\scM$ is a closed, full-dimensional, connected subspace.  To see this, note that axiom~\ref{it:R1} of \cref{def:multimatroid-rank} ensures that $\scM \subseteq \R^{\decoseto}$.  Inequalities~\ref{it:R2} and \ref{it:R3} describe $\scM$ as an intersection of closed half-spaces, so $\scM$ is closed and convex (in particular, connected).  The fact that it is full-dimensional follows from \cref{lem:non-empty}.
\end{definition}

As mentioned in the introduction, a key feature of multimatroids is that their restriction to any colored set yields a matroid, so a multimatroid can in some sense be viewed as a way of patching together a collection of ordinary matroids.  More precisely, if $\rkk$ defines a multimatroid and $S \in \decoset$, then every subset of $S$ is also an element of $\decoset$, so one can define a rank function on the power set $\calP(S)$ by the restriction of $\rkk$, and it is straightforward to verify from conditions~\ref{it:BR1} -- \ref{it:BR3} that $\rkk|_{\calP(S)}$ defines a matroid on the ground set $S$ (the axiom~\ref{it:BR4} is irrelevant for this purpose).

A similar story holds for $\R$-multimatroids, but the restrictions are the following weakening of matroids; these are essentially the same as polymatroids \cite{Edmonds}, but with real-valued rather than integer-valued rank function.

\begin{definition} An {\bf $\R$-matroid} is a function $\rk \colon \calP(S) \to \R$ on the power set $\calP(S)$ satisfying
\begin{enumerate}[label=(M\arabic*)]
 \item \label{it:M1} $\rk(\emptyset)=0$;
 \item \label{it:M2} $\rk(X) \leq \rk(Y)$ whenever $X \subseteq Y$;
 \item \label{it:M3} for every $X,Y \in \calP(S)$, $\rk(X \cup Y) + \rk(X \cap Y) \leq \rk(X) + \rk(Y)$.
\end{enumerate}
\end{definition}

Comparing the conditions \ref{it:R1}, \ref{it:R2} and \ref{it:R3} of \cref{def:multimatroid-rank} with \ref{it:M1}, \ref{it:M2} and \ref{it:M3} above, one sees that the following definition indeed yields an $\R$-matroid.

\begin{definition}
    \label{lem:M(S)}  For an $\R$-multimatroid $\bfM$ and $S \in \decoset$, the {\bf restriction of $\bfM$ to $S$} is the $\R$-matroid $\bfM(S)$ given by restricting the rank function of $\bfM$ to subsets of $S$.
\end{definition}

\subsection{Independence polytopal complexes} \label{sec:IPC}

We now introduce the analogue for multimatroids and $\R$-multimatroids of the independence polytope of a matroid.  Carrying forth the philosophy that a multimatroid $\bfM$ is a way of patching together a collection of matroids $\bfM(S)$, we form the independence polytopal complex of $\bfM$ by patching together the independence polytopes of the matroids $\bfM(S)$.  This same idea holds for $\R$-multimatroids, but in this case the components $\bfM(S)$ are only $\R$-matroids, so care must be taken in how their independence polytopes are defined.

\begin{definition}
    The \textbf{independence polytope of an $\R$-matroid $\bfM(S)$} is the polytope
    \begin{equation} \label{eq:IP-ABD} \IP(\bfM(S)) \coloneqq  \left\lbrace \sum_{i \in S} x_i \ove_i  \in \overline{\R}^S_{\geq 0} \; \left\vert \; \sum_{i \in X} x_i \leq \rkk(X) \text{ for all } X \subseteq S \right.\right\rbrace \subseteq \overline{\R}^S_{\geq 0} \subseteq N^{\pi}_\R.
\end{equation}\end{definition}

\begin{remark} If $\bfM(S)$ is an honest matroid and not merely an $\R$-matroid, then one defines the independent sets of $\bfM(S)$ as those subsets $I \subseteq S$ such that $\rkk(I)=|I|$. In this situation,   \cite[equation (4)]{ABD} show that the independence polytope of $\bfM(S)$ is given by
        \[\IP(\bfM(S)) =  \conv \{ \ove_{I} \; | \; I \subseteq S \text{ an independent set} \},
    \] 
which explains its name.  However, when $\bfM(S)$ is $\R$-valued, there is not, to our knowledge, a good notion of independent sets (as discussed in \cref{rem:Rmultimatroid}), so we do not know of a convenient description of $\IP(\bfM(S))$ as a convex hull.
\end{remark}

Gluing the polytopes $\bfM(S)$ across all colored subsets $S \in \decoset$, one obtains the following.

\begin{definition}
\label{def:IPC} 
The \textbf{independence polytopal complex} of an $\R$-multimatroid $\bfM$ is the union
\[ \IPC(\bfM) \coloneqq  \bigcup_{S \in \decoset} \IP(\bfM(S)).
\]
\end{definition}

\begin{remark} \label{rmk:IPCM-pc}
    The fact that $\IPC(\bfM)$ forms a polytopal complex follows from the observation that, for any  $S_1, S_2 \in \decoset$ with $S_1 \cap S_2 \neq \emptyset$, one has
    \[\IP(\bfM(S_1)) \cap \IP(\bfM(S_2)) = \IP(\bfM(S_1 \cap S_2))\]
    by axiom~\ref{it:R2}, and therefore this intersection is a face of each of the two independence polytopes on the left-hand side.
    This furthermore shows that 
    \begin{equation}
        \label{eq:IPCmax}
    \IPC(\bfM) = \bigcup_{T \in \decosetmax} \IP(\bfM(T))
    \end{equation}
    since for every $S \subseteq T$ we have $\IP(\bfM(S)) \subseteq \IP(\bfM(T))$. 
\end{remark}

\begin{example}
\label{ex:IPC1}
For the multimatroid $\bfM$ as in \cref{ex:M12bar12}, $\IPC(\bfM)$ can be realized as the polytopal complex in $\R^2$ depicted in Figure~\ref{fig:IPC1}.\end{example}

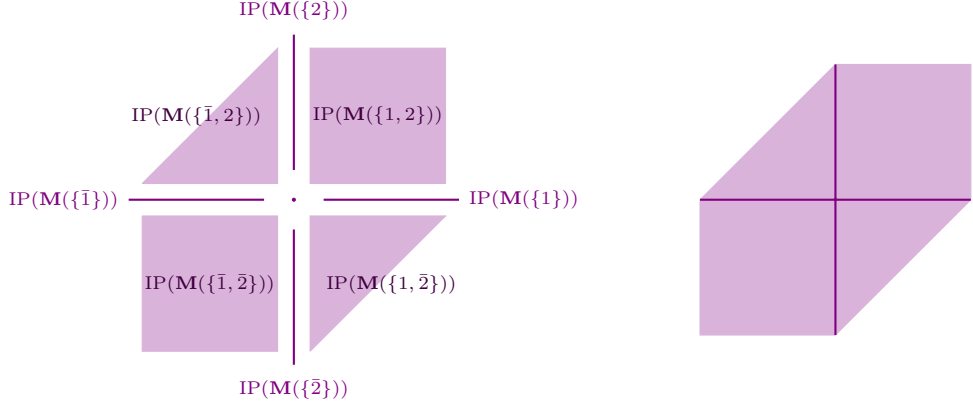
\begin{figure}[h!]
\centering
\begin{tikzpicture}[xscale=1.8, yscale=1.8]
\draw [ao(english),fill=ao(english),opacity=0.3] (1,0) -- (1,1) -- (0,1) -- (-1,0) -- (-1,-1) -- (0,-1) -- (1,0);

\draw [ao(english), thick]
(0,0) -- (1,0)
(0,0) -- (-1,0)
(0,0) -- (0,1)
(0,0) -- (0,-1);
\begin{scope}
[shift={(-4,0)}]

\begin{scope}
[shift={(0.12, 0.12)}]
\draw [ao(english),fill=ao(english),opacity=0.3] (1,0) -- (1,1) -- (0,1) -- (0,0) -- (1,0);
\node [ao(english)!50!black] at (0.5,0.5) {\scriptsize $\IP(\bfM(\{1,2\}))$}; \end{scope}

\begin{scope}
[shift={(-0.12, -0.12)}]
\draw [ao(english),fill=ao(english),opacity=0.3] (-1,0) -- (-1,-1) -- (0,-1) -- (0,0) -- (-1,0);
\node [ao(english)!50!black] at (-.5,-.5) {\scriptsize $\IP(\bfM(\{\bar{1},\bar{2}\}))$};\end{scope}

\begin{scope}
[shift={(-0.12, 0.12)}]
\draw [ao(english),fill=ao(english),opacity=0.3] (-1,0) -- (0,0) -- (0,1) -- (-1,0);
\node [ao(english)!50!black] at (-.6,.5) {\scriptsize $\IP(\bfM(\{\bar{1},2\}))$};\end{scope}

\begin{scope}
[shift={(0.12, -0.12)}]
\draw [ao(english),fill=ao(english),opacity=0.3] (0,-1) -- (0,0) -- (1,0) -- (0,-1);
\node [ao(english)!50!black] at (.6,-.5) {\scriptsize $\IP(\bfM(\{1,\bar{2}\}))$};\end{scope}

\draw [ao(english), thick] 
(0.22,0) -- (1.22,0)
(-0.22,0) -- (-1.22,0)
(0,0.22) -- (0,1.22)
(0,-0.22) -- (0,-1.22);
\node [ao(english)] at (1.7,0) {\scriptsize $\IP(\bfM(\{1\}))$};
\node [ao(english)] at (-1.7,0) {\scriptsize $\IP(\bfM(\{\bar{1}\}))$};
\node [ao(english)]at (0,1.4) {\scriptsize $\IP(\bfM(\{2\}))$};
\node [ao(english)] at (0,-1.4) {\scriptsize $\IP(\bfM(\{\bar{2}\}))$};

\filldraw [ao(english)] (0,0) circle[radius=0.111mm];
\end{scope}
\end{tikzpicture}
\caption{The independence polytopal complex of the multimatroid $\bfM$ of \cref{ex:M12bar12}}.
\label{fig:IPC1}
\end{figure}

\begin{example}
\label{ex:IPC2}
For the $\R$-multimatroid $\bfM$ as in \cref{ex:Rmultimatroid}, the complex $\IPC(\bfM)$ is depicted in Figure~\ref{fig:IPC2}. \end{example}
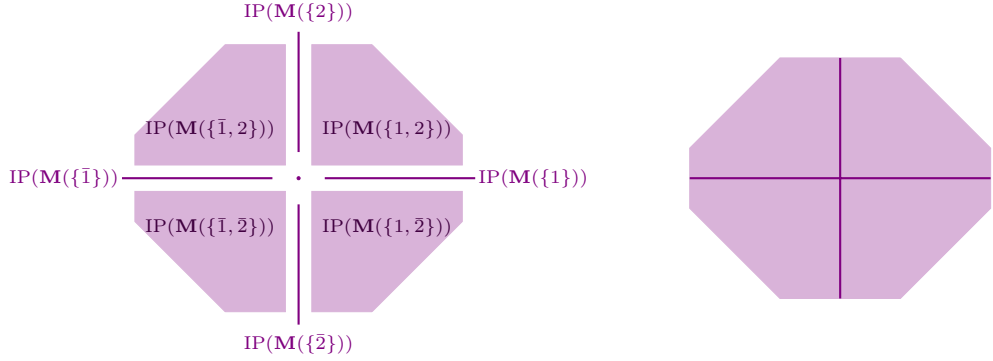
\begin{figure}[h!]
\centering
\begin{tikzpicture}[xscale=0.4, yscale=0.4]
\draw [ao(english),fill=ao(english),opacity=0.3] (2,4) -- (5,1) -- (5,-1) -- (2,-4) -- (-2,-4) -- (-5,-1) -- (-5,1) -- (-2,4) -- (2,4);

\draw [ao(english), thick] (0,0) -- (5,0) (0,0) -- (-5,0) (0,0) -- (0,4) (0,0) -- (0,-4) ;
\begin{scope}
[shift={(-18,0)}]

\begin{scope}
[shift={(0.44, 0.44)}]\draw [ao(english),fill=ao(english),opacity=0.3] (2,4) -- (5,1) -- (5,0) -- (0,0) --(0,4) -- (2,4);
\node [ao(english)!50!black] at (2.5,1.2) {\scriptsize $\IP(\bfM(\{1,2\}))$}; \end{scope}

\begin{scope}
[shift={(-0.44, -0.44)}]\draw [ao(english),fill=ao(english),opacity=0.3] (-2,-4) -- (-5,-1) -- (-5,0) -- (0,0) --(0,-4) -- (-2,-4);
\node [ao(english)!50!black] at (-2.5,-1.2) {\scriptsize $\IP(\bfM(\{\bar{1},\bar{2}\}))$};\end{scope}

\begin{scope}
[shift={(-0.44, 0.44)}]\draw [ao(english),fill=ao(english),opacity=0.3] (-2,4) -- (-5,1) -- (-5,0) -- (0,0) --(0,4) -- (-2,4);
\node [ao(english)!50!black] at (-2.5,1.2) {\scriptsize $\IP(\bfM(\{\bar{1},2\}))$};\end{scope}

\begin{scope}
[shift={(0.44, -0.44)}]\draw [ao(english),fill=ao(english),opacity=0.3] (2,-4) -- (5,-1) -- (5,0) -- (0,0) --(0,-4) -- (2,-4);
\node [ao(english)!50!black] at (2.5,-1.2) {\scriptsize $\IP(\bfM(\{1,\bar{2}\}))$};\end{scope}

\draw [ao(english), thick] 
(0.87,0) -- (5.87,0)
(-0.87,0) -- (-5.87,0)
(0,0.87) -- (0,4.87)
(0,-0.87) -- (0,-4.87);
\node [ao(english)] at (7.8,0) {\scriptsize $\IP(\bfM(\{1\}))$};
\node [ao(english)] at (-7.8,0) {\scriptsize $\IP(\bfM(\{\bar{1}\}))$};
\node [ao(english)]at (0,5.5) {\scriptsize $\IP(\bfM(\{2\}))$};
\node [ao(english)] at (0,-5.5) {\scriptsize $\IP(\bfM(\{\bar{2}\}))$};

\filldraw [ao(english)] (0,0) circle[radius=0.5mm];
\end{scope}
\end{tikzpicture}
\caption{The independence polytopal complex of the multimatroid $\bfM$ of \cref{ex:Rmultimatroid}}.
\label{fig:IPC2}
\end{figure}

\begin{remark}
    \label{rem:M(T)polytope}
    Analogously to the observation in \cref{rem:M(T)fan} that $\Sigma^\pi$ can be viewed as a union of copies of the $n$-dimensional \annoyingfanname, with one copy associated to each $T \in \decosetmax$, equation \eqref{eq:IPCmax} shows that $\IPC(\bfM)$ can be viewed as a union of independence polytopes of matroids on size-$n$ ground sets, with one matroid associated to each $T \in \decosetmax$.  This parallelism is no accident: we will see in \cref{lem:C(DM)=IPC(M)} that, under a certain condition on $\bfM$, the polytopal complex $\IPC(\bfM)$ is a normal complex of $\Sigma^\pi$.
\end{remark}

In the same way that the volume of a normal complex $C_{\Sigma^\pi, \ast}(D)$ was given in \eqref{eq:volpi} as the sum over volumes of its components in each $\overline{\mathbb{R}}^T$, we define the volume of $\IPC(\bfM)$ as the sum
\begin{equation} \label{eq:volIPC}  \vol(\IPC(\bfM))  \coloneqq  \sum_{T \in \decosetmax} \vol_{T}( \IP(\bfM(T))),
\end{equation}
where the volume $\vol_T$ on $\overline{\R}^T$ is defined by \eqref{eq:VolT}.  One key property of this volume function that will play a crucial role below is the following.

\begin{lemma} \label{lem:volispoly} The volume function
\begin{align*}
    &\vol \colon \scM \to \R\\
    &\vol(\bfM) = \vol(\IPC(\bfM))
\end{align*}
is a polynomial function on $\scM \subseteq \R^{\decoset}$.
\end{lemma}

\begin{proof} It is enough to show that, for every $T \in \decosetmax$, the expression $\Vol_T(\IP(\bfM(T)))$ is a polynomial in $\{\rkk(S)\}_{\emptyset \neq S\subseteq T}$.
This follows by applying \cite[Theorem 13.4.4]{cox2011toric} to the {stellahedral fan}, for which the nef divisors correspond to the monotone submodular rank functions by \cite[ Proposition 3.13]{EHL}.
\end{proof}

\subsection{Divisor associated to a multimatroid} 

Having fully defined the objects appearing on the right-hand side of \cref{thm:B}, we now turn to the left-hand side, describing how an $\R$-multimatroid $\bfM$ defines a  divisor $D_{\bfM}$ on $\Sigma^\pi$. In particular, using the notation of \cref{sec:fan}, we set
\begin{equation}  \label{eq:DM}
D_{\bfM} \coloneqq  \sum_{S \in \decoset} \rkk(S) x_S.
\end{equation}
We say that an $\R$-multimatroid is {\bf pseudo-cubical} if $D_\bfM$ is a pseudo-cubical divisor on $\Sigma^\pi$ under the inner product described in \cref{sec:NCcolfan}.  By \cref{lem:DMpseudocubical}, this is equivalent to the condition that, for every maximal chain $\scC =(S_1 \subsetneq \dots \subsetneq S_n)$ of nonempty colored sets, we have
 \begin{equation}\label{eq:DMpseudo} 2 \rkk(S_i) \geq  \rkk(S_{i-1}) + \rkk(S_{i+1}) \qquad \text{ and } \qquad \rkk(S_{j}) \geq \rkk(S_{j-1})
\end{equation} for each $i \in [n-1]$ and $j \in [n]$, where $\rkk(S_0)=0$. 
 Similarly, an $\R$-multimatroid $\bfM$ is {\bf cubical} if $D_{\bfM}$ is a cubical divisor, which is equivalent to the condition that the inequalities \eqref{eq:DMpseudo} are all strict.

It is straightforward to check from the conditions \eqref{eq:DMpseudo} that the Boolean multimatroid is pseudo-cubical for any $(E,\pi)$.  The following examples illustrate the behavior of $D_{\bfM}$ in some somewhat more interesting cases.

\begin{example} \label{ex:M12bar12-DM} For $\bfM$ as in \cref{ex:M12bar12}, we have
\[ D_\bfM = 2 \left(x_{\{1,2\}} +  x_{\{\bar{1},\bar{2}\}} \right) + 1 \left( x_{\{1,\bar{2}\}} + x_{\{\bar{1},2\}} + x_{\{1\}} + x_{\{\bar{1}\}} + x_{\{2\}} + x_{\{\bar{2}\}} \right).
\]
This is a pseudo-cubical but not cubical multimatroid, since, for instance, the chain
\[\scC = (\{1\} \subseteq \{1,2\})\]
yields the equality $2 \cdot 1 = 0 + 2$.
\end{example}

\begin{example}
\label{ex:cubicalDM}
For $\bfM$ as in \cref{ex:Rmultimatroid}, we have
\[D_{\bfM} = 4\left(x_{\{1\}} +x_{\{\bar{1}\}} \right) +5\left(x_{\{2\}} +x_{\{\bar{2}\}} \right) + 6\left(x_{\{1,2\}} + x_{\{\bar{1}, 2\}} + x_{\{1, \bar{2}\}}  +x_{\{\bar{1}, \bar{2}\}} \right) .  \]
This is a cubical $\R$-multimatroid; for instance, the chain
\[\scC = (\{1\} \subseteq \{1,2\})\]
yields the inequality $2 \cdot 4 > 0 + 6$, and the chain 
\[\scC = (\{2\} \subseteq \{1,2\})\]
yields the inequality $2 \cdot 5 > 0 + 6$.
\end{example}

Interestingly, the pseudo-cubical condition on $D_{\bfM}$ in fact implies the $\R$-multimatroid axioms.  This observation is useful in what follows, so we prove it in the following lemma.

 \begin{lemma}
     \label{lem:PCimpliesRM}
     For any function $\rkk \colon \decoset \rightarrow \R$ such that $\rkk(\emptyset) = 0$, if the divisor $\sum_{S \in \decoset} \rkk(S) x_S$ is pseudo-cubical, then $\rkk$ defines an $\R$-multimatroid.
\end{lemma}
\begin{proof}
    The second inequality of \eqref{eq:DMpseudo} clearly implies the monotonicity axiom~\ref{it:R2}. We now show that the submodularity axiom~\ref{it:R3} is implied by the first inequality of \eqref{eq:DMpseudo}. 
    
    For every $I \in \decoset$, and for every $x,y \in E\setminus I$ such that $I \cup \{x,y\} \in \decoset$, applying \eqref{eq:DMpseudo} gives
    \[ 2 \rkk(I \cup \{x\}) \geq \rkk(I) + \rkk(I \cup \{x,y\}) \qquad \text{and} \qquad 2 \rkk(I \cup \{y\}) \geq \rkk(I) + \rkk(I \cup \{x,y\}),
    \]
    so we obtain
    \begin{equation} \label{eq:idea} \rkk(I \cup \{x\}) + \rkk(I \cup \{y\}) \geq \rkk(I) + \rkk(I \cup \{x,y\}).
    \end{equation} This shows that axiom~\ref{it:R3} holds when $|S \cup T|=|S|+1=|T|+1$, by setting $I = S \cap T$.
    It is known that the validity of \ref{it:R3} in such cases implies its validity in full generality \cite[Theorem 44.1]{Schrijver}.  For the sake of self-containment, we include a proof below.
    
    Again setting $I = S\cap T$, denote $S=I \cup \{s_1, \dots, s_k\}$ and $T=I \cup \{t_1, \dots, t_\ell\}$. Then, for every $0 \leq a \leq k$ and $0 \leq b \leq \ell$, define the set
    \[ X_{a,b} \coloneqq I \cup \{s_1,\dots, s_a, t_1, \dots, t_b\} \in \decoset,
    \]
    where $X_{k,0} \coloneqq S$, $X_{0,\ell} \coloneqq T$, and $X_{0,0} \coloneqq I$. Then \eqref{eq:idea} yields
\[ \rkk(X_{a,b-1}) + \rkk(X_{a-1,b}) \geq  \rkk(X_{a-1,b-1})+\rkk(X_{a,b})
\]
for all $a \in [k]$ and $b \in [\ell]$.  Taking the sum over all such $a$ and $b$, we obtain the inequality
\[ \rkk(S) + \rkk(T) = \rkk(X_{k,0}) + \rkk(X_{0,\ell}) \geq \rkk(X_{0,0}) +\rkk(X_{k,\ell})= \rkk(S \cap T)+\rkk(S \cup T),
\] which shows that \ref{it:R3} holds and thus concludes the proof.
\end{proof}

Now, in the topological space $\scM$ of $\R$-multimatroids defined in \cref{sec:scM}, define the subset
\begin{align*}
\scMc &\coloneqq  \left\lbrace\bfM \in \scM \; | \; \bfM \text{ is cubical}\right\rbrace.
\end{align*}
Then we have the following key properties.

\begin{lemma}\label{lem:non-empty}  \label{cor:cubical-open} The space $\scMc$ is a nonempty, open subset of $\R^{\decoseto}$.  In particular, $\scM$ has nonempty interior and is therefore a full-dimensional subset of $\R^{\decoseto}$.\end{lemma}

\proof The conditions \eqref{eq:DMpseudo} with strict inequalities manifestly define $\scMc$ as an open subset of $\R^{\decoseto}$. We are left to show that $\scMc$ is nonempty. To do so, we define a specific cubical $\R$-multimatroid $\bfM$ by setting
\begin{equation}
    \label{eq:PCRmultimatroid}
{\rkk}(S) \coloneqq  \binom{n+1}{2} - \binom{n+1-|S|}{2}
\end{equation}
for each $S \in \decoset$.  It is straightforward to see that the inequalities \eqref{eq:DMpseudo} hold and therefore $\bfM$ is cubical.  By \cref{lem:PCimpliesRM}, it follows that $\bfM$ is an $\R$-multimatroid, so $\bfM \in \scMc$.
\endproof

\begin{remark} 
The fact that \eqref{eq:PCRmultimatroid} defines an $\R$-multimatroid and not an ordinary multimatroid is one of the key reasons why we require the generalization from multimatroids to $\R$-multimatroids in this work.  In fact, if $n \geq 3$, then no $(E,\pi)$ can admit a cubical multimatroid.  To see this, let $\bfM$ be a multimatroid on $(E,\pi)$ with $n \geq 3$, so that, for every maximal chain $\scC = (S_1 \subsetneq \dots \subsetneq S_n)$ in $\decoseto$, one has $\rkk(S_i) \in \N$ and $\rkk(S_{i}) \leq \rkk(S_{i+1}) \leq \rkk(S_i)+1$ for every $i \in [n-1]$. The inequalities \eqref{eq:DMpseudo} show that $\bfM$ can only be cubical if
\[ \text{(i)} \quad  2\rkk(S_1) > \rkk(S_2) \qquad \text{and} \qquad  \text{(ii)} \quad 2\rkk(S_2)  > \rkk(S_1) + \rkk(S_3).\]
Condition (i) implies that $\rkk(S_1)\neq 0$ and so necessarily $\rkk(S_1)=1$. This implies that $\rkk(S_2)=1$ and so $\rkk(S_3) \in \{1,2\}$. None of these options is compatible with condition (ii), and so $\bfM$ cannot be cubical.
\end{remark}

At this point, we have all the ingredients necessary to prove \cref{thm:B}, so we turn in the next section to its proof and then, in turn, to the deduction of \cref{thm:h-integrals}.

\section{Proofs of main theorems}
\label{sec:proofs}

While \cref{thm:B} was stated in the introduction as a statement about multimatroids, we in fact prove it for all $\R$-multimatroids.  The statement is as follows.

\begin{thm}\label{thm:multimatroid_integral}   For any $\R$-multimatroid $\bfM$ on $(E,\pi)$,
\begin{equation} 
\label{eq:thm51}\int_{\Sigma^\pi} \left( D_\bfM\right)^n = \vol(\IPC(\bfM)). \end{equation}
\end{thm}

\subsection{Proof of Theorem~\ref{thm:multimatroid_integral}}

When $\bfM$ is pseudo-cubical, \cref{thm:multimatroid_integral} can be deduced from what we have already done, so we begin with this case.

\begin{lemma}
\label{lem:C(DM)=IPC(M)}
Let $\bfM$ be a pseudo-cubical $\R$-multimatroid.  Then 
\begin{equation} \label{eq:NC=IPC} \NC_{\Sigma^\pi, \ast}(D_\bfM) = \IPC(\bfM),
\end{equation}
and furthermore, their volumes agree in the sense that
\begin{equation} \label{eq:VolC=VolIPC}
    \vol_{\pi}(\NC_{\Sigma^\pi, \ast}(D_\bfM)) = \vol(\IPC(\bfM)),
\end{equation} 
where the left-hand side is defined by \eqref{eq:volpi} and the right-hand side by \eqref{eq:volIPC}.  In particular, combining \eqref{eq:VolC=VolIPC} with \cref{cor:NR23-application-Sr}, it follows that \cref{thm:multimatroid_integral} holds when $\bfM$ is pseudo-cubical.
\end{lemma} 
\begin{proof}
As noted in \eqref{eq:NCSigmapi}, we have
\[\NC_{\Sigma^{\pi}, \ast}(D_\bfM) = \bigcup_{T \in \decosetmax} \bigcup_{\scC \in \mchain(T)} P_{\sigma_{\scC, \ast}} (D_\bfM),\]
where $\mchain(T)$ again denotes the set of maximal chains $\scC = (S_1 \subsetneq \cdots \subsetneq S_n)$ of colored sets with $S_n = T$.  Expanding the definition of $P_{\sigma_{\scC, \ast}} (D_\bfM)$ as in \eqref{eq:Psigma} we obtain
\[P_{\sigma_{\scC}, \ast} (D_\bfM) = \left\{\vec{x} \in \sigma_{\scC} \; | \; \vec{x} \ast \ove_S \leq \rkk(S) \; \text{ for all } S \in \scC \right\}.\]
We claim, in fact, that
\[P_{\sigma_{\scC}, \ast} (D_\bfM) =  \{\vec{x} \in \sigma_{\scC} \; | \; \vec{x} \ast \ove_S \leq \rkk(S) \; \text{ for all } S \subseteq T\}.\]
Because the proof of this claim is somewhat cumbersome, we relegate it to \cref{lem:(2.9)} below.  Assuming it, and writing $\vec{x} \in \sigma_{\scC}$ in terms of the orthonormal basis $\{\ove_i\}_{i \in T}$ for $\overline{\R}^T$, we find
\begin{equation} \label{eq:UPC} \bigcup_{\scC \in \mchain(T)} P_{\sigma_{\scC}, \ast} (D_\bfM) =  \left\{\sum_{i \in T} x_i \ove_i \in \overline{\R}^T_{\geq 0} \; \left| \; \sum_{i \in S} x_i \leq \rkk(S) \; \text{ for all } S \subseteq T\right.\right\}.
\end{equation}
In view of \eqref{eq:IP-ABD}, we conclude that \eqref{eq:UPC} coincides with $\IP(\bfM(T))$, so taking the union over $T \in \decosetmax$ we see that the equality \eqref{eq:NC=IPC} holds.  The fact that the notions of volume agree is the content of equations \eqref{eq:volpi} and \eqref{eq:volIPC}, so \eqref{eq:VolC=VolIPC} holds, as well.
\end{proof}

\begin{lemma}
    \label{lem:(2.9)}
    Let $\bfM$ be a pseudo-cubical $\R$-multimatroid.  Then
    \[P_{\sigma_{\scC}, \ast} (D_\bfM) =  \{\vec{x} \in \sigma_{\scC} \; | \; \vec{x} \ast \ove_S \leq \rkk(S) \; \text{ for all } S \subseteq T\}\]
    for any $T \in \decosetmax$ and any $\scC \in \mchain(T)$.
\end{lemma}
\begin{proof}
Fix $T \in \decosetmax$ and $\scC \in \mchain(T)$.  Up to relabeling, we can write $T = [n]$ and
\[\scC = ([1] \subsetneq [2] \subsetneq \cdots \subsetneq [n]).\]  Now, let $\vec{x} \in P_{\sigma_{\scC}, \ast}(D_\bfM)$.  By definition, this means that $\vec{x} \in \sigma_{\scC} = \cone(\ove_{[1]}, \ove_{[2]}\ldots, \ove_{[n]})$, so
\[\vec{x} = a_1 \ove_{[1]} + \cdots + a_n \ove_{[n]}= (a_1 + \cdots + a_n) \ove_1 + (a_2 + \cdots + a_n)\ove_2 + \cdots + a_n \ove_n\]
for some $a_1, \ldots, a_n \geq 0$.  Equivalently, $\vec{x} = c_1 \ove_1 + \cdots + c_n\ove_n$ for some $c_1 \geq c_2 \geq \cdots \geq c_n \geq 0$.  The defining inequalities of $P_{\sigma_{\scC}, \ast}(D_\bfM)$ then imply that 
\begin{equation}
    \label{eq:(2.9)goal}
\sum_{i \in S} c_i \leq \rkk(S)
\end{equation}
whenever $S = [j]$ for some $j \in [n]$, and what we must prove is that the same is true for all $S \subseteq [n]$.

To prove this, we define a pair of functions
\[f: \calP([n]) \rightarrow \calP([n]) \;\; \text{ and } \;\; g: \calP([n]) \rightarrow \calP([n])\]
on the power set $\calP([n])$ by setting $f([j]) = g([j]) = [j]$ for each $j \in [n]$, and setting
\[f(S) = S \cup \{a_S\} \;\; \text{ and } \;\; g(S) = S \setminus \{b_S\}\]
for any $S$ that is not of this form, where $a_S$ and $b_S$ are defined by
\begin{align*}
    &a_S \coloneqq \min\{i \in [n] \; | \; i \notin S\}, &&b_S \coloneqq \min\{i \in S \; | \; i > a_S\}.
\end{align*}
We first claim that, if \eqref{eq:(2.9)goal} holds for $f(S)$ and $g(S)$, then it also holds for $S$.  This is trivially true when $S = [j]$ for some $j$, so we prove it in the case where $S$ is not of this form.  The assumption that \eqref{eq:(2.9)goal} holds for $f(S)$ means that $c_{a_S} + \sum_{i \in S} c_i \leq \rkk(f(S))$.  Since $a_S < b_S$ and therefore $c_{a_S} \geq c_{b_S}$, it follows that
\begin{equation}
    \label{eq:assumption1}
    c_{b_S} + \sum_{i \in S} c_i \leq \rkk(f(S)).
\end{equation}
On the other hand, the assumption that \eqref{eq:(2.9)goal} holds for $g(S)$ means that
\begin{equation}
    \label{eq:assumption2}
    -c_{b_S} + \sum_{i \in S} c_i \leq \rkk(g(S)).
\end{equation}
Adding \eqref{eq:assumption1} and \eqref{eq:assumption2} yields 
\[2\sum_{i \in S} c_i \leq \rkk(f(S)) + \rkk(g(S)) \leq 2\rkk(S),\]
where the second inequality follows by applying the pseudo-cubicality condition \eqref{eq:DMpseudo} to any maximal chain containing $g(S) \subsetneq S \subsetneq f(S)$.  This shows that \eqref{eq:(2.9)goal} holds for $S$, as claimed.

From here, we prove that \eqref{eq:(2.9)goal} holds for all $S \subseteq [n]$ by descending induction, first on $a_S$ and then on $b_S$.  Note that it suffices to prove the claim when $S$ is not of the form $[j]$ for any $j \in [n]$, since it holds by assumption when $S$ is of this form.  The base case of the first induction is the case $a_S = n-1$.  Then $f(S) = [n]$ and $g(S) = [n-2]$, so \eqref{eq:(2.9)goal} holds for both of these and therefore holds for $S$.  Suppose, then, that \eqref{eq:(2.9)goal} holds for all $S' \subseteq [n]$ with $a_{S'}>k$, and let $S$ be such that $a_S = k$.  We now introduce the second induction, a descending induction on $b_S$.

The base case is $b_S = n$.  In this case, $f(S)$ has $a_{f(S)}> k$, so \eqref{eq:(2.9)goal} holds for $f(S)$ by the first inductive hypothesis, whereas $g(S) = [k-1]$, so \eqref{eq:(2.9)goal} holds for $g(S)$ by assumption.  Therefore, \eqref{eq:(2.9)goal} holds for $S$, completing the base case.  Finally, suppose that \eqref{eq:(2.9)goal}  holds for all $S'' \subseteq [n]$ with $b_{S''} > \ell$, and let $S$ be such that $b_S = \ell$.  Then $f(S)$ has $a_{f(S)}>k$ and $g(S)$ has $b_{g(S)} > \ell$, so \eqref{eq:(2.9)goal} holds for both of these by the two inductive hypotheses.  This implies that \eqref{eq:(2.9)goal} holds for $S$, completing the proof.
\end{proof}

\begin{remark}
\cref{lem:(2.9)} shows that the fan $\Sigma^\pi$ satisfies the global condition of \cite[equation (2.9)]{NR23}.
\end{remark}

\begin{remark} We note that the pseudo-cubical condition is critical in order for Lemmas~\ref{lem:C(DM)=IPC(M)} and \ref{lem:(2.9)} to hold.  One indication of this, as pointed out in \cref{rmk:IPCM-pc}, is that $\IPC(\bfM)$ is always a polytopal complex, whereas $\NC_{\Sigma^\pi, \ast}(D_\bfM)$ is not necessarily a polytopal complex unless $\bfM$ is pseudo-cubical.

For a specific example, consider the $\R$-multimatroid $\bfM$ on $E = \{1\} \sqcup \{2\}$ defined by
\[ \rkk(\{1\}) = 2, \qquad  \rkk(\{2\}) = 1, \qquad\rkk(\{1,2\}) = 3.\]
The figure below illustrates $C_{\Sigma^\pi, \ast}(D_{\bfM})$ on the left and $\IPC(\bfM)$ on the right.

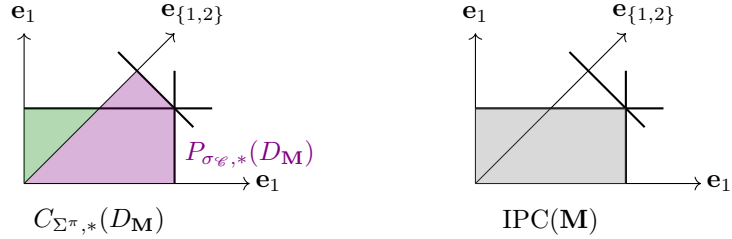
\begin{figure}[h!]
\centering
    \begin{tikzpicture}[scale=1]
    \draw[->] (0, 0) -- (3,0);
    \draw[->] (0,0) -- (0,2);
    \draw[->] (0,0) -- (2, 2);
    \draw[thick] (2,0) -- (2,1.5);
    \draw[thick] (0,1) -- (2.5,1);
    \draw[thick] (1.25, 1.75) -- (2.25, 0.75);
    \fill[fill=ao(english), opacity=0.3] (0,0) -- (2,0) -- (2,1) -- (1.5,1.5) -- (0,0);
    \fill[fill=ao(green), opacity=0.3] (0,0) -- (1,1) -- (0,1) -- (0,0);
    \node at (3.25,0) {$\e_{1}$};
    \node at (0,2.25) {$\e_{1}$};
    \node at (2.25,2.25) {$\e_{\{1,2\}}$};
    \node at (1,-0.5) {$C_{\Sigma^\pi, \ast}(D_{\bfM})$};
    \node at (3,0.4) {\textcolor{ao(english)}{$P_{\sigma_{\scC}, \ast}(D_{\bfM})$}}; 
    \begin{scope}[shift={(6,0)}]
    \draw[->] (0, 0) -- (3,0);
    \draw[->] (0,0) -- (0,2);
    \draw[->] (0,0) -- (2,2);
    \draw[thick] (2,0) -- (2,1.5);
    \draw[thick] (0,1) -- (2.5,1);
    \draw[thick] (1.25, 1.75) -- (2.25, 0.75);
    \fill[fill=gray, opacity=0.3] (0,0) -- (2,0) -- (2,1) -- (0,1) -- (0,0);
        \node at (3.25,0) {$\e_{1}$};
    \node at (0,2.25) {$\e_{1}$};
    \node at (2.25,2.25) {$\e_{\{1,2\}}$};
    \node at (1,-0.5) {$\IPC(\bfM)$};
    \end{scope}
    \end{tikzpicture}
\end{figure}
Note that in the cone $\sigma_{\scC}$ corresponding to $\scC = (\{1\} \subseteq \{1,2\})$, the purple shaded polytope  $P_{\sigma_{\scC}, \ast}(D_{\bfM})$ on the left is bounded only by the two hyperplanes normal to the two incident rays.  On the other hand, $\IPC(\bfM)$ is bounded by all three hyperplanes normal to the rays in $\overline{\R}^T_{\geq 0}$.  Thus, we see that $C_{\Sigma^\pi, \ast}(D_\bfM) \neq \IPC(\bfM)$ in this example, and moreover, that Lemma~\ref{lem:(2.9)} fails: imposing the defining equalities $\vec{x} \ast \ove_S \leq \rkk(S)$ on each maximal cone separately (which defines $C_{\Sigma^\pi, \ast}(D_\bfM)$) does not coincide with imposing them simultaneously on all of $\R^T_{\geq 0}$ (which defnes $\IPC(\bfM)$).
\end{remark}
 
 To illustrate \cref{lem:C(DM)=IPC(M)} in the pseudo-cubical case, it is illuminating to look back at Examples~\ref{ex:IPC1} and \ref{ex:IPC2}.  In particular, one can see visually that the independence polytopal complexes in both of these examples are normal complexes of the fan $\Sigma^\pi$ illustrated in Figure~\ref{fig:SigmaB2}, and that \cref{ex:IPC2} is cubical while \cref{ex:IPC1} is not; this is consistent with the computations of Examples~\ref{ex:M12bar12-DM} and \ref{ex:cubicalDM}.  To see the same phenomenon in higher dimension, we turn to the following example.

\begin{example} 
Suppose that $n=3$, and let $\bfM$ be the cubical $\R$-matroid defined by \eqref{eq:PCRmultimatroid}.  Then, for any $T \in \decosetmax$, the independence polytope $\IP(\bfM(T))$ is the polytope illustrated in Figure~\ref{fig:cubicalIPC}.  In particular, note that each maximal cone of $\Sigma^\pi$ contains exactly one vertex of the complex, which is consistent with the claim that $\IPC(\bfM)$ is a normal complex of $\Sigma^\pi$ associated to a cubical divisor. 
    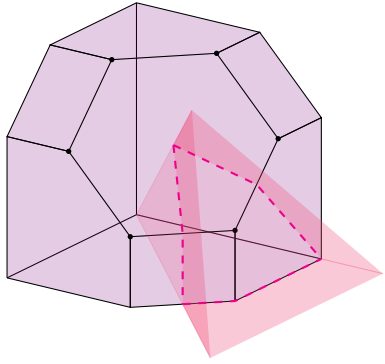
\begin{figure}[h!]
        \centering
\tdplotsetmaincoords{70}{125} 

\begin{tikzpicture}[tdplot_main_coords,scale=1]
  \draw (0,0,0) -- (0,0,3);
  \draw (0,0,0) -- (0,3,0);
  \draw (0,0,0) -- (3,0,0);
  
    \draw [ao(english), fill=ao(english),opacity=0.2] (1,2,3) -- (2,1,3) -- (3,1,2) -- (3,2,1) -- (2,3,1) -- (1,3,2)-- (1,2,3);
    \draw [black, fill=ao(english), fill opacity=0.2] (1,3,2) -- (0,3,2) -- (0,2,3) -- (1,2,3) -- (1,3,2);
    \draw [black, fill=ao(english), fill opacity=0.2](2,1,3)--(2,0,3) --(3,0,2) -- (3,1,2) -- (2,1,3);
    \draw [black, fill=ao(english), fill opacity=0.2](2,3,1) -- (2,3,0) -- (3,2,0) -- (3,2,1) --(2,3,1);
    \draw [black, fill=ao(english), fill opacity=0.2](0,0,3) -- (0,2,3) -- (1,2,3) -- (2,1,3) -- (2,0,3)--(0,0,3);
    \draw [black, fill=ao(english), fill opacity=0.2](0,3,0) -- (0,3,2) -- (1,3,2) -- (2,3,1) -- (2,3,0)--(0,3,0);
    \draw [black, fill=ao(english), fill opacity=0.2](3,0,0) -- (3,0,2) -- (3,1,2) -- (3,2,1) -- (3,2,0)--(3,0,0);

    \fill[black] (1,2,3) circle (1pt);
    \fill[black] (2,1,3) circle (1pt);
    \fill[black] (1,3,2) circle (1pt);
    \fill[black] (2,3,1) circle (1pt);
    \fill[black] (3,1,2) circle (1pt);
    \fill[black] (3,2,1) circle (1pt);

    \fill [magenta!70!orange, opacity=0.1] 
    (0,0,0) -- (2,2,2) -- (2.5,2.5, 1) -- (2.5,2.5,0) -- (0,0,0);
     \fill [magenta!70!orange, opacity=0.1] (0,0,0) -- (2,2,2) -- (1.5,3,1.5) -- (0,3,0) -- (0,0,0);
     \fill [magenta!70!orange, opacity=0.1] (0,0,0) -- (2.5,2.5, 0) -- (2,3,0) -- (0,3,0) -- (0,0,0);
    
    \fill [magenta!70!orange, opacity=0.3] (2,2,2) -- (3,3,3) -- (4,4,0) -- (2.5,2.5,0) -- (2.5,2.5,1) -- (2,2,2);
    \fill [magenta!70!orange, opacity=0.3] (2,2,2) -- (3,3,3) -- (0,4,0) -- (0,3,0) -- (1.5,3,1.5) -- (2,2,2);
    \fill [magenta!70!orange, opacity=0.3] (2.5,2.5,0) -- (2,3,0) -- (0,3,0) -- (0,4,0) -- (4,4,0) -- (2.5,2.5,0);

    \draw [dashed, thick, magenta] 
    (2,2,2) -- (2.5,2.5, 1) -- (2.5,2.5,0)
    (2,2,2) -- (1.5,3,1.5) -- (0,3,0)
    (2.5,2.5, 0) -- (2,3,0) -- (0,3,0); 

    \draw [magenta!70!orange, opacity=0.4]
    (0,3,0)--(0,4,0)
    (2.5, 2.5, 0) -- (4,4,0)
    (2,2,2)--(3,3,3);

    \draw [magenta!70!orange, opacity=0.2]
    (0,0,0) -- (0,3,0)
    (0,0,0) -- (2.5, 2.5, 0)
    (0,0,0) -- (2,2,2);
    
\end{tikzpicture}
        \caption{The polytope $\IP(\bfM(T))$ for $\bfM$ as in \eqref{eq:PCRmultimatroid} and any $T \in \decosetmax$, with one of the maximal cones of $\Sigma^\pi$ shaded.} 
        \label{fig:cubicalIPC}
    \end{figure}
\end{example}

Having proven \cref{thm:multimatroid_integral} in the pseudo-cubical case, the general case follows almost immediately by polynomiality considerations. 

\begin{proof}[Proof of \cref{thm:multimatroid_integral}] Letting $\bfM$ vary in $\scM$, both sides of the theorem can be viewed as functions on $\scM$---that is, functions on the parameters $\left(\rkk(S)\right)_{S \in \decoseto}$.  The left-hand side is manifestly polynomial in these parameters, and the right-hand side is polynomial as well by \cref{lem:volispoly}.  \Cref{lem:C(DM)=IPC(M)} shows that the theorem holds when $\bfM$ is pseudo-cubical, so it in particular holds when $\bfM$ is cubical.  Therefore, the two sides of \cref{thm:multimatroid_integral} agree on the subset $\scMc \subseteq \scM$, which is nonempty and open by \cref{lem:non-empty}.  Polynomiality then implies that they agree on all of $\scM$, proving the result.
\end{proof}

The proof of \cref{thm:multimatroid_integral} shows that it can actually be seen as an identity between polynomials functions in the parameters $\left(\rkk(S)\right)_{S \in \decoseto}$ on the space $\R^{\decoseto}_{\geq 0}$.  For values of these parameters that do not necessarily define an $\R$-multimatroid, the left-hand side of the theorem manifestly still makes sense; as for the right-hand side, note that \eqref{eq:IPCmax} can be taken as the definition of $\IPC(\bfM)$, and while it may not be a polytopal complex for values of $\left(\rkk(S)\right)_{S \in \decoseto}$ that do not satisfy the $\R$-multimatroid axioms, it is nevertheless a union of polytopes with finite volume.  In the following subsection, we rephrase this equality of polynomials in a different basis for $\R^{\decoseto}$, which is what will ultimately allow us to deduce \cref{thm:h-integrals}.

\subsection{An alternative formulation}

Recall from \eqref{eq:hS} that there is an alternative set of generators $h_S$ for $A^*(\Sigma^\pi)$, and the divisors $x_S$ can be expressed in terms of the divisors $h_S$ via \cref{lem:hStoxS}.  Thus, we can rewrite
\begin{equation}
    \label{eq:DMaS}
D_{\bfM} = \sum_{S \in \decoseto} a_S h_S
\end{equation}
for coefficients $\aaS = (a_S)_{S \in \decoseto}$.  In particular, the left-hand side of \cref{thm:multimatroid_integral} is equal to the value at this particular choice of $\aaS$ of the polynomial
\[I(\aaS) \coloneqq  \int_{\Sigma^{\pi}}  \left(\sum_{S \in \decoseto} a_S h_S\right)^n.\]
The right-hand side of \cref{thm:multimatroid_integral} can also be expressed as a value of a polynomial in $\aaS$, and this polynomial turns out to have a very nice description.  To state this description, for any $T \in \decosetmax$ and any $S \subseteq T$, define the simplex
\[
\Delta^T_S \coloneqq \conv \bigg(\{\mathbf{0}\}\cup \{\ove_j \; | \; j\in S\}\bigg) \subseteq \overline{\R}^T.\]
Then the intersection of $\IPC(\bfM)$ with $\overline{\R}^T$ can be expressed as a Minkowski sum of these simplices, as the following proposition verifies.

\begin{proposition}\label{cor:IPC(M)=Delta} Let $\bfM$ be an $\R$-multimatroid. Then 
\[\Vol(\IPC(\bfM)) = \sum_{T \in \decosetmax} \Vol_T\left( \sum_{S \in \decoseto} a_{S} \Delta^T_{S \cap T}\right),\]
where the sum denotes the Minkowski sum of polytopes and the coefficients $\aaS = (a_S)_{S \in \decoseto}$ are defined by \eqref{eq:DMaS}.
\end{proposition}

Before proving the proposition, we illustrate it in some examples.

\begin{example}
    Consider the $\R$-multimatroid in \cref{ex:Rmultimatroid}, whose independence polytopal complex is illustrated in \cref{ex:IPC1}.  Via the change of coordinates from $\{x_S\}$ to $\{h_S\}$ in \cref{lem:hStoxS}, we obtain
    \[D_{\bfM} = h_{1\bar{2}} + h_{\bar{1}2},\]
    so $a_{1\bar{2}} = a_{\bar{1}2} = 1$ and $a_S = 0$ for all other $S \in \decoseto$.  Thus, taking $T = \{1,2\}$ on the right-hand side of \cref{cor:IPC(M)=Delta}, one obtains the following Minkowski sum of polytopes:
    \[1\Delta^T_{\{1, \bar{2}\} \cap \{1,2\}} + 1\Delta^T_{\{\bar{1}, 2\} \cap \{1,2\}} = 1\Delta^T_{\{1\}} + 1 \Delta^T_{\{2\}}.\]
    The contribution to the proposition from this $T$ then follows from the equality
    \[\IPC(\M(\{1,2\})) = 1\Delta^T_{\{1\}} + 1 \Delta^T_{\{2\}},\]
    which can be seen from \cref{ex:IPC1} because $\IPC(\M(\{1,2\}))$ is a square with vertices at $\mathbf{0}$, $\ove_1$, $\ove_2$, and $\ove_1 + \ove_2$.  On the other hand, a similar computation shows that the contribution to the proposition from  $T = \{\bar{1},2\}$ is
    \[\IPC(\M(\{\bar{1},2\})) = 1 \Delta^T_{\{\bar{1},2\}},\]
    which can again be seen from \cref{ex:IPC1} because $\IPC(\M(\{\bar{1},2\}))$ is the standard simplex in its quadrant.
\end{example}

\begin{example} \label{ex:IPCsum}
    Now consider the $\R$-multimatroid in \cref{ex:Rmultimatroid}, whose independence polytopal complex is illustrated in \cref{ex:IPC2}.  Again applying \cref{lem:hStoxS}, we find
\[D_\bfM= -1 \left(h_{1}+h_{\bar{1}}\right) - 2 \left(h_{2} + h_{\bar{2}}\right) + 3 \left(h_{\{1,2\}} + h_{\{\bar{1},2\}}+ h_{\{1,\bar{2}\}}+ h_{\{\bar{1},\bar{2}\}}\right).\]
    Hence, the Minkowski sum of polytopes on the right-hand sum of \cref{cor:IPC(M)=Delta} for $T = \{1,2\}$ is
    \begin{align*}
    &-1 \Delta_{\{1\} \cap \{1,2\} }^T -2 \Delta_{\{2\} \cap \{1,2\}}^T + 3 \Delta_{\{1,\bar{2}\} \cap \{1,2\}}^T + 3 \Delta_{\{\bar{1},2\} \cap \{1,2\}}^T + 3 \Delta_{\{1,2\}}^T\\
    =&-1 \Delta_{\{1\}}^T -2 \Delta_{\{2\}}^T + 3 \Delta_{\{1\}}^T + 3 \Delta_{\{2\}}^T + 3 \Delta_{\{1,2\}}^T\\
    =& \;2 \Delta_{\{1\}}^T + 1 \Delta_{\{2\}}^T + 3 \Delta_{\{1,2\}}^T.
    \end{align*}
    Thus, the statement of the proposition for $T = \{1,2\}$ is that 
    \[\IP(\bfM(\{1,2\})) = 2 \Delta_{\{1\}}^T + 1 \Delta_{\{2\}}^T + 3 \Delta_{\{1,2\}}^T,
    \]
    which one indeed sees is the case from the figure below.
            
            
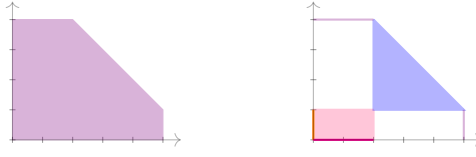
\begin{figure*}[h!]
        \centering
\begin{tikzpicture}[xscale=0.4, yscale=0.4]
  \foreach \i in {0,1,...,5} \draw [opacity=0.4] (\i,-0.1)--(\i,.1);
  \foreach \i in {0,1,...,4} \draw [opacity=0.4] (-0.1,\i)--(.1,\i);
  
\draw [ao(english),fill=ao(english),opacity=0.3] (2,4) -- (5,1) -- (5,0) -- (0,0) -- (0,4) -- (2,4);
\draw [opacity=0.3, ->] (0,0) --(5.6,0);
\draw [opacity=0.3, ->] (0,0) --(0,4.6);
\begin{scope}[shift={(10,0)}]
  \foreach \i in {0,1,...,5} \draw [opacity=0.4] (\i,-0.1)--(\i,0.1);
  \foreach \i in {0,1,...,4} \draw [opacity=0.4] (-0.1,\i)--(0.1,\i);
    \draw [thick, orange!50!magenta,fill=orange!50!magenta,opacity=0.3] (2,1) -- (2,0) -- (0,0) -- (0,1) -- (2,1);
    \draw [thick, blue,fill=blue,opacity=0.3]  (2,4) -- (2,1) -- (5,1)--(2,4);
    \draw [thick, ao(english), opacity=0.3] (5,1)--(5,0);
    \draw [thick, ao(english), opacity=0.3] (2,4)--(0,4);
    \draw [thick, orange] (0,0) -- (0,1);
    \draw [thick, magenta] (0,0) -- (2,0);
\draw [opacity=0.3, ->] (0,0) --(5.6,0);
\draw [opacity=0.3, ->] (0,0) --(0,4.6);
\end{scope}
\end{tikzpicture}    \end{figure*}
\noindent A similar decomposition applies in this example for all other $T \in \decosetmax$.
\end{example}

Equipped with the intuition of these examples, we are prepared to prove the proposition in general.

\begin{proof}[Proof of \cref{cor:IPC(M)=Delta}]
By the definition of $\Vol(\IPC(\bfM))$ in \eqref{eq:volIPC}, it suffices to prove that, for any $T \in \decosetmax$, one has
\begin{equation}
\label{eq:IPMT=MinkowskiSum}
    \IP(\bfM(T)) = \sum_{S \in \decoseto} a_{S} \Delta^T_{S \cap T}.
\end{equation}
This follows from the computation of the independence polytope of a matroid as a Minkowski sum of simplices given in \cite{ABD}.  Before stating their result, we require some notation. Note, either using \cref{lem:hStoxS} or through \cite{ABD}, that the free module $\Z[\{x_S\}_{S\subseteq T}]$ is isomorphic to $\Z[\{h_S^T\}_{S\subseteq T}]$, where \[h_S^T : = \sum_{\substack{S' \subseteq T\\ S \cap S' \neq \emptyset}} x_{S'}.\] Thus, one can write \[D_{\bfM(T)} \coloneqq  \sum_{S \subseteq T} \rkk(S) x_S= \sum_{S\subseteq T} a_S^T h_S^T\] for uniquely defined  $a_S^T \in \Z$.  In fact, \cite[Proposition 4.3]{ABD} shows that
\begin{equation}
\label{eq:ABD}
\IP(\bfM(T)) = \sum_{S \subseteq T} a_S^T \Delta^T_S,
\end{equation}
so the content of the proof of \eqref{eq:IPMT=MinkowskiSum} is relating the coefficients $a_S^T$ to the coefficients $a_S$.  We can unpack this relationship using the commutative diagram
\[\begin{tikzcd} [row sep=3mm, column sep=2cm]   
 \Z[\{h_S\}_{S\in \decoseto}] \arrow[r] \arrow[ddd, "\cong"] & \Z[\{h_S^T\}_{S\subseteq T}] \arrow[ddd, "\cong"] \\
 \\        \\
 \Z[\{x_S\}_{S\in \decoseto}] \arrow[r] & \Z[\{x_S\}_{S \subseteq T}],
 \end{tikzcd}
 \]
where the upper horizontal arrow is $h_S \mapsto h^T_{S \cap T}$ and the lower horizontal arrow is
\[x_S \mapsto \begin{cases} x_S & \text{ if } S \subseteq T\\ 0 & \text{otherwise.}\end{cases}\]
In particular, applying the commutativity of this diagram to $\sum_{S \in \decoseto} a_Sh_S$ shows that
\[\sum_{S \subseteq T} \rkk(S) x_S = \sum_{S \in \decoseto} a_S h^T_{S \cap T}.\]
Re-writing the right-hand side slightly gives
\[\sum_{S \subseteq T} \rkk(S) x_S = \sum_{S \subseteq T} \left( \sum_{\substack{S' \in \decoseto\\ S' \cap T = S}} a_{S'} \right)h^T_{S},\]
from which we deduce
\begin{equation*}
a_S^T = \sum_{\substack{S' \in \decoseto\\ S' \cap T = S}} a_{S'}.
\end{equation*}
Combining this with \eqref{eq:ABD} shows \eqref{eq:IPMT=MinkowskiSum} and thus completes the proof.
\end{proof}

In light of \cref{cor:IPC(M)=Delta}, we define the polynomial
\[V(\aaS) \coloneqq  \sum_{T \in \decosetmax} \Vol_T\left( \sum_{S \in \decoseto} a_{S} \Delta^T_{S \cap T}\right)\]
in the parameters $\aaS$.  The equality of \cref{thm:multimatroid_integral}, when translated into these parameters, now becomes the following.

 \begin{thm}
 \label{thm:h-integrals-v2}
\Cref{thm:multimatroid_integral} is equivalent to the equality of polynomials $I(\aaS) = V(\aaS)$.
 \end{thm}
 \begin{proof}
As above, fix an $\R$-multimatroid $\bfM$ and write $D_\bfM = \sum_{S \in \decoseto} a_S h_S$. Then the left-hand side of \eqref{eq:thm51} is equal to $I(\aaS)$ by definition, and \cref{cor:IPC(M)=Delta} shows that the right-hand side of \eqref{eq:thm51} is equal to $V(\aaS)$.  In particular, \cref{thm:multimatroid_integral} is equivalent to the statement that $I(\aaS) = V(\aaS)$ on $\scM$, the subset of $\R^{\decoseto}$ consisting of values of the parameters $\aaS$ that define an $\R$-multimatroid.  Since both sides are polynomial and $\scM$ contains an open subset of $\R^{\decoseto}$, this is equivalent to the corresponding equality on the entirety of $\R^{\decoseto}$.
 \end{proof}

\subsection{Proof of Theorem~\ref{thm:h-integrals}}

Having re-expressed \cref{thm:multimatroid_integral} as an equality of polynomials in this way, it follows that the coefficient of any monomial in $I(\aaS)$ agrees with the corresponding coefficient in $V(\aaS)$.  The coefficient of a monomial in $I(\aaS)$ is, by definition, an integral of a monomial in the $h_S$'s.
On the other hand, the coefficient of the corresponding monomial in $V(\aaS)$ is the mixed volume of certain simplices, for which a previously-known formula is recorded as the following lemma.

\begin{lemma}\label{lem:HallRado}
For a subset $S\subseteq [n]$, define $\Delta_S \subseteq \R^n$ as the convex hull of $\{\mathbf 0\} \cup \{\mathbf e_i \; | \; i\in S\}$.  Then, for subsets $S_1, \dotsc, S_n$ of $[n]$ (with repetitions allowed), the mixed volume $\text{MV}$ of the corresponding simplices is given by
\[
\text{MV}(\Delta_{S_1}, \dotsc, \Delta_{S_n}) = \begin{cases}
1 & \text{there exists a bijection $\iota: [n] \to [n]$ such that $\iota(i) \in S_i$ for each $i$}\\ 
0 & \text{otherwise}.
\end{cases}
\]
\end{lemma}

\begin{proof}
Standard results in toric geometry translate mixed volumes to intersection numbers of nef divisors \cite[Chapter 5.4]{FultonToric}.  Applying this to polystellahedral fans, as done in \cite[Section 2.2]{EL}, one finds that the lemma is a restatement of \cite[Theorem 1.3 and Lemma 5.2]{EL}.  One can also deduce the lemma from \cite[Theorem 5.1]{Postnikov} or from \cite[Theorem A(b)]{EFLS}.
\end{proof}

We are now ready to prove \cref{thm:h-integrals}, whose statement we recall for convenience.

\begin{customthm}{A}
For any collection $S_1, \ldots, S_n \in \decoset$ (with repetitions allowed), we have
\begin{equation}
    \label{eq:TheoremAgoal} \tag{A}
\int_{\Sigma^{\pi}} h_{S_1} \cdots h_{S_n} = |\T_\pi(S_1, \ldots, S_n)|,
\end{equation}
where
\[\T_\pi(S_1, \ldots, S_n) \coloneqq  \left\{ T \in \decosetmax \; \left| \; \begin{matrix} \text{there exists a bijection } \iota:[n] \rightarrow T\\ \text{with } \iota(i) \in S_i \text{ for each }i \end{matrix} \right.\right\}.\]
\end{customthm}
\begin{proof}
The left-hand side of \eqref{eq:TheoremAgoal} is the coefficient of the monomial $a_{S_1} \cdots a_{S_n}$ in the polynomial $I(\aaS)$.  The coefficient of the same monomial in $V(\aaS)$ is, by definition, the sum of mixed volumes
\[\sum_{T \in \decosetmax} \text{MV}\left(\Delta^T_{S_1 \cap T}, \ldots, \Delta^T_{S_n \cap T}\right).\]
For each $T \in \decosetmax$, \Cref{lem:HallRado} states that 
\[\text{MV}\left(\Delta^T_{S_1 \cap T}, \ldots, \Delta^T_{S_n \cap T}\right)= \begin{cases} 1 & \text{ if } T \in \T_\pi(S_1, \ldots, S_n)\\ 0 & \text{ otherwise},\end{cases}\]
so the coefficient of $a_{S_1} \cdots a_{S_n}$ in $V(\aaS)$ is precisely the right-hand side of \eqref{eq:TheoremAgoal}.  Thus, the two sides agree by \cref{thm:multimatroid_integral} and \cref{thm:h-integrals-v2}.
\end{proof}

\section{Intersection numbers of psi-classes}
\label{sec:psiclasses}

One way in which to understand the special role played by the generators $h_S$ in the Chow ring of $\Sigma^\pi$ is to look more closely at the uniform case, in which case $\Sigma^\pi$ is the fan $\Sigma^r_n$ studied in \cite{clader2022wonderful}.  In this section, we prove that the generators $h_S$ in the uniform case are pullbacks of psi-classes under certain forgetful morphisms, analogously to the results of \cite{DR} for the case of Losev--Manin space.  This allows us to reprove some cases of \cref{thm:h-integrals} from a more geometric perspective, assuming some familiarity with the tools and language of moduli of curves.  Throughout what follows, we assume that $\pi$ is uniform with $|E_i|= r$ for each $i$, so we can write
\begin{equation}
    \label{eq:Euniform}
E = \{1^0, 1^1, \ldots, 1^{r-1}\} \sqcup \{2^0, 2^1, \ldots, 2^{r-1}\} \sqcup \cdots \sqcup \{n^0, n^1, \ldots, n^{r-1}\},
\end{equation}
and we assume that $r \geq 2$.

\subsection{Background on the moduli space}

The papers \cite{clader2022permutohedral, clader2022wonderful} study the moduli space $\L^r_n$ parameterizing the following data:
\begin{itemize}
    \item an {\bf $r$-pinwheel} curve $C$, which is a rational curve consisting of a central projective line from which $r$ chains of projective lines (called {\bf spokes}) emanate;
    \item an order-$r$ automorphism $\sigma$ of $C$;
    \item a pair of distinct fixed points $x^\pm \in C$ of $\sigma$;
    \item $n$ labeled $r$-tuples $(z_1^j)_{j \in \Z_r}, \ldots, (z_n^j)_{j \in \Z_r}$ of points $z_i^j \in C$ (called {\bf light points}) satisfying
    \[\sigma(z_i^j) = z_i^{j+1 \!\!\!\mod r}\]
    for each $i,j$, which are allowed to coincide with one another and with $x^{\pm}$;
    \item an additional labeled $r$-tuple $(y^\ell)_{\ell \in \Z_r}$ satisfying
    \[\sigma(y^\ell) = y^{\ell+1 \!\!\! \mod r}\]
    for each $\ell$, whose elements are distinct from one another as well as from $x^\pm$ and $z_i^j$.
\end{itemize}
These marked points are subject to a stability condition, the details of which can be found in \cite[Section 2.1]{clader2022permutohedral}; see Figure~\ref{fig:Lrn} for an example element.  Note that via the expression \eqref{eq:Euniform} we can view the light points as indexed by elements of $E$, and the light points on any given spoke form a colored set.

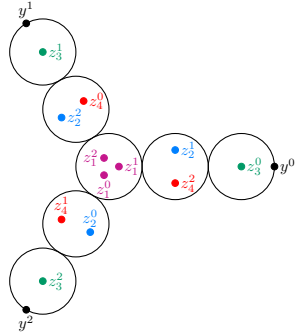
\begin{figure}[h!]
\begin{center}
    \begin{tikzpicture}[scale=0.88,  every node/.style={scale=0.6}]
    \filldraw[magenta!80!blue] (0.15,0) circle (0.05cm) node[right]{$z^1_1$};
    \filldraw[magenta!80!blue] (-0.075,0.1299) circle (0.05cm) node[left]{$z^2_1$};
    \filldraw[magenta!80!blue] (-0.075,-0.1299) circle (0.05cm) node[below]{$z^0_1$};
    \filldraw[blue!50!cyan](1,0.25) circle (0.05cm) node[right]{$z^1_2$};
    \filldraw[blue!50!cyan](-0.7165, 0.741) circle (0.05cm) node[right]{$z^2_2$};
    \filldraw[blue!50!cyan](-0.283, -0.991) circle (0.05cm) node[above]{$z^0_2$};
    \filldraw[red](1,-0.25) circle (0.05cm) node[right]{$z^2_4$};
    \filldraw[red](-0.383,0.991) circle (0.05cm) node[right]{$z^0_4$};
    \filldraw[red](-0.717,-0.8) circle (0.05cm) node[above]{$z^1_4$};
    \filldraw[blue!40!green] (2,0) circle (0.05cm) node[right]{$z^0_3$};
    \filldraw[blue!40!green] (-1,1.732) circle (0.05cm) node[right]{$z^1_3$};
    \filldraw[blue!40!green] (-1,-1.732) circle (0.05cm) node[right]{$z^2_3$};
    \draw (0,0) circle (0.5cm);
    \draw (1,0) circle (0.5cm);
    \draw (2,0) circle (0.5cm);
    \filldraw (2.5,0) circle (0.05cm) node[right]{$y^0$};
    \draw (-0.5,0.866) circle (0.5cm);
    \draw (-1,1.732) circle (0.5cm);
    \filldraw (-1.25,2.165) circle (0.05cm) node[above]{$y^1$};
    \draw (-0.5,-0.866) circle (0.5cm);
    \draw (-1,-1.732) circle (0.5cm);
    \filldraw (-1.25,-2.165) circle (0.05cm) node[below]{$y^2$};
        \end{tikzpicture}
\end{center}
\caption{A sample element of $\L^3_4$, where each circle represents a projective line and $\sigma$ is the rotational automorphism.  Not pictured are the marked points $x^+$ and $x^-$, which are the two fixed points of $\sigma$ and must both lie on the central component.}
\label{fig:Lrn}
\end{figure}

When $r=2$, the moduli space $\L^r_n$ is the toric moduli space constructed by Batyrev--Blume \cite{BB1, BB2}, which is the toric variety $X_{B_n}$ associated to the type-$B$ permutohedral fan $\Sigma^2_n$.  When $r >2$, on the other hand, the moduli space is no longer toric, so in particular it no longer coincides with $X_{\Sigma^r_n}$.  Nevertheless, the main result of \cite{clader2022wonderful} is that $\L^r_n$ can be viewed as a wonderful compactification (the closure of a very affine variety) inside $X_{\Sigma^r_n}$, and that the inclusion induces an isomorphism
\[A^*(\Sigma^r_n) \cong A^*(\L^r_n).\]
Furthermore, for any chain $\scC$ of colored subsets of $E$, the class of the torus-invariant stratum in $X_{\Sigma^r_n}$ corresponding to the cone $\sigma_{\scC}$ restricts to a {\bf boundary stratum} $\overline{S}_{\scC} \subseteq \Lrn$. Roughly, if $\scC$ is a chain of length $k$, then $\overline{S}_{\scC}$ is the closure of the locally closed subvariety $S_{\scC}$ consisting of curves in which each spoke has length $k$ and the distribution of marked points is specified by $\scC$; see \cite[Section 4]{clader2022permutohedral} for the precise definition.

In particular, the generator $x_S \in A^*(\Sigma^r_n)$ restricts to the boundary divisor $[X_S] \in A^*(\Lrn)$, which is the class of the closure of the locus $X_S$ of curves in which each spoke has length one and the light marked points on the $y^0$-spoke are precisely those indexed by $S$.  The generators $h_S$, on the other hand, restrict to pullbacks of certain psi-classes.  To explain this, we introduce the moduli space $\Mbar^1_S$, which parameterizes the following data:
\begin{itemize}
    \item a curve $C$ that consists of a chain of projective lines;
    \item a pair of distinct points $x^\pm \in C$;
    \item a tuple $(z_i)_{i \in S}$ of points $z_i \in C$ (again called light points), which are allowed to coincide with one another and with $x^\pm$;
    \item an additional marked point $y \in C$, which is not allowed to coincide with any other marked point.
\end{itemize}
These marked points, again, are subject to a stability condition made precise in \cite[Section 3.1]{clader2022permutohedral}; to state it succinctly, one can view $\Mbar^1_S$ as a Hassett space $\Mbar_{0, \bfw}$ with weight vector
\[\bfw = \left(\frac{1}{2} + \epsilon, \frac{1}{2} + \epsilon, \underbrace{\epsilon, \ldots, \epsilon}_{|S| \text{ copies}}, 1\right)\]
for $0<\epsilon<1/(|S|+2)$.  This, in particular, forces that $x^\pm$ lie together on one end-component of the chain $C$ and $y$ lies on the opposite end-component.

\begin{remark}
\label{rem:M(T)geometric}
We briefly digress to notice that, for any $T \in \decosetmax$, the toric variety $X_{\Sigma_T}$ associated to the \annoyingfanname (defined in \cref{rem:M(T)fan}) can be identified with an open subvariety $U_T \subseteq \Mbar^1_T$.
When one observes that $\Mbar^1_T$ can be identified with the toric variety of the stellahedral fan,
this open inclusion is the inclusion of toric varieties corresponding to the inclusion of the \annoyingfanname, in the stellahedral fan.
To provide a modular description, for any element of $\Mbar^1_T$, let $C_0 \subseteq C$ be the component containing $x^{\pm}$, and choose coordinates $C_0$ in which $x^+ = 1$, $x^- = -1$, and the unique node of $C_0$ (or the point $y$, if there is no node) is equal to $0$.  Define $U_T$ to consist of those curves for which, in these coordinates, one has $z_i \neq \infty$ for each $i$.  Via these coordinates, $U_T$ can be viewed as an iterated blow-up of $\A^n$ along the loci where $k$ coordinates are equal to zero, in decreasing order from $k=n$ to $k=2$, and this gives an identification $U_T = X_{\Sigma_T}$.

In light of this observation, the inclusion of fans $\Sigma_T \hookrightarrow \Sigma^r_n$ induces an inclusion
\[U_T \hookrightarrow X_{\Sigma^r_n},\]
and the intersection of $U_T$ with $\Lrn \hookrightarrow X_{\Sigma^r_n}$ is the union of the locally closed boundary strata $S_{\scC}$ for $\scC$ a chain of subsets of $T$.  As $T$ varies, these intersections $U_T \cap \L^r_n$ cover $\Lrn$, so this observation can be viewed as the moduli-theoretic analogue of the covering of $\Sigma^\pi$ by the fans $\Sigma_T$ described in \cref{rem:M(T)fan} and the covering of $\IPC(\bfM)$ by the polytopes $\IP(\bfM(T))$ described in \cref{rem:M(T)polytope}.
\end{remark}

Returning to the definition of $h_S$ via psi-classes, recall that one can define psi-classes on any Hassett space as the first Chern classes of the cotangent line bundles at the marked points. 
In the case of $\Mbar^1_S$, this in particular yields a class
\[\psi_y \coloneqq c_1(\mathbb{L}_y) \in A^1(\Mbar^1_S),\]
where $\mathbb{L}_y$, roughly speaking, is the line bundle on $\Mbar^1_S$ whose fiber at a marked curve $C$ is the cotangent line to $C$ at the point $y$.  More precisely, $\mathbb{L}_y = \sigma_y^*\omega_{\mathcal{C}^1_S/\Mbar^1_S}$, where $\mathcal{C}^1_S \rightarrow \Mbar^1_S$ is the universal curve and $\sigma_y: \Mbar^1_S \rightarrow \mathcal{C}^1_S$ is the section corresponding to the marked point $y$.

Now, any $S \in \decoseto$ induces a morphism
\[F_S \colon \L^r_n \rightarrow \Mbar^1_{S}\]
that forgets all of the marked points in $C$ except for $x^\pm$, $y^0$, and the light marked points $z_i^j$ indexed by $i^j \in S$. Equipped with the morphism $F_S$, we claim that the generators $h_S$ can be described as follows.

\begin{lemma}
\label{lem:hS-pullback-psi}
    Under the isomorphism $A^*(\Sigma^r_n) \cong A^*(\L^r_n)$ given by $x_S \mapsto [X_S]$, one has
    \[h_S = F_S^*(\psi_y).\]
\end{lemma}
\begin{proof}
Similarly to the boundary divisors on $\L^r_n$, there are boundary divisors $D^1_T \in A^1\left(\Mbar^1_S\right)$ for each nonempty subset $T \subseteq S$.  Namely, $D^1_T$ is the class of closure of the locus of curves consisting of two components, one containing the marked points $y$ and $z_i$ with $i \in T$, and the other containing the remaining marked points. 

We claim, first, that
\begin{equation}
    \label{eq:psiyBD}
\psi_y = \sum_{\emptyset \neq T \subseteq S} D^1_T.
\end{equation}
The proof is by induction on $|S|$.   When $|S|= 0$, the moduli space $\Mbar^1_S \cong \Mbar_{0, 3}$ is a single point, so both sides of \eqref{eq:psiyBD} are zero for dimension reasons.  Suppose, now, that \eqref{eq:psiyBD} holds on $\Mbar^1_{S'}$ with $|S'| = |S|-1$.  To prove that it holds on $\Mbar^1_{S}$, consider the forgetful map
\[f\colon \Mbar^1_{S} \rightarrow \Mbar^1_{S \setminus \{\star\}}\]
given by forgetting one of the light points $z_\star$ and stabilizing.  A standard comparison argument (see e.g., \cite[Lemma 1.3.1]{PSI}) shows that the following equation holds in $A^*\left(\Mbar^1_{S}\right)$:
\begin{equation}
   \label{eq:psiycomparison}
f^*\psi_y = \psi_y - D^1_{\{\star\}},
\end{equation}
in which the $\psi_y$ on the left-hand side lies on $\Mbar^1_{S\setminus \{\star\}}$ and the $\psi_y$ on the right-hand side lies on $\Mbar^1_{S}$.  (The idea of the proof of \eqref{eq:psiycomparison} is that the line bundles $\mathbb{L}_y$ and $f^*\mathbb{L}_y$ agree away from the locus of curves that are stabilized under $f$, which is precisely the subvariety whose class is $D^1_{\{\star\}}$.) On the other hand, the boundary divisors are related under $f$ by
\begin{equation}
    \label{eq:DTcomparison}
f^*D^1_T = D^1_T + D^1_{T \cup \{\star\}}
\end{equation}
for any nonempty subset $T \subseteq S \setminus \{\star\}$.  Pulling back both sides of the equation \eqref{eq:psiyBD} on $\Mbar^1_{S \setminus \{\star\}}$ under $f$ and applying equations \eqref{eq:psiycomparison} and \eqref{eq:DTcomparison}, one obtains
\[\psi_y -D^1_{\{\star\}} = \sum_{\emptyset \neq T \subseteq S \setminus \{\ast\}} D^1_T +  D^1_{T \cup \{\star\}}. \]
Rearranging this equation yields the equation \eqref{eq:psiyBD} on $\Mbar^1_S$.

Having proven \eqref{eq:psiyBD}, we deduce the lemma by pulling back both sides under $F_S\colon \L^r_n \rightarrow \Mbar^1_S$.  Namely, it is straightforward from the definitions of the boundary divisors to see that 
\[F_S^* (D^1_T) = \sum_{R \cap S = T} [X_R],\]
so \eqref{eq:psiyBD} yields
\[F_S^*(\psi_y) = \sum_{\emptyset \neq T \subseteq S} \sum_{R \cap S = T} [X_R] = \sum_{R \cap S \neq \emptyset} [X_R],\]
which is precisely $h_S$ under the identification of $[X_R] \in A^1(\Lrn)$ with $x_R \in A^1(\Sigma^r_n)$.
\end{proof}

\begin{remark}
    One might wonder why we do not also consider classes $F_S^*(\psi_{x^\pm})$ or $F_S^*(\psi_{z_i})$, which can also be defined for any $S$.  But in fact, one can show that
    \[F_S^*(\psi_{x^\pm}) = 0 \;\;\; \text{ and } \;\;\; F_S^*(\psi_{z_i}) = -h_{\{i\}} \text{ for any } i \in S,\]
    so no new divisors on $\Lrn$ are obtained in this way.  The key point in the proof of these observations is that the analogues for $\psi_{x^\pm}$ and $\psi_{z_i}$ of equation \eqref{eq:psiycomparison} are
    \[f^*\psi_{x^\pm} = \psi_{x^\pm} \;\;\; \text{ and } \;\;\; f^*\psi_{z_i} = \psi_{z_i},\]
    since the marked points $x^\pm$ and $z_i$ never lie on a component contracted by $f$.  Iterating this observation and using that $\psi_{x^\pm} = 0$ on $\Mbar^1_\emptyset$ for dimension reasons shows that $\psi_{x^\pm} = 0$ on every $\Mbar^1_S$. On the other hand, one can check (for example, using \cite[Lemma 2.7]{Sharma}) that $\psi_{z_i} = -\psi_y$ on $\Mbar^1_{\{1\}}$, from which the equation $F_S^*(\psi_{z_i}) = -h_{\{i\}}$ follows.
\end{remark}

\subsection{Geometric perspective on Theorem~\ref{thm:h-integrals}}
\label{subsec:psiclass}

Given the perspectives on $x_S$ via boundary divisors on $\Lrn$ and $h_S$ via psi-classes, one can prove at least some cases of \cref{thm:h-integrals} using geometric techniques from the study of moduli of curves.  Although we were not able to prove \cref{thm:h-integrals} in full generality using these techniques, we believe that it is an illuminating perspective that deserves further exploration, so we illustrate the ideas in this last subsection.  Throughout what follows, we identify $x_S$ and $h_S$ with their images under the isomorphism $A^*(\Sigma^r_n) \cong A^*(\Lrn)$, so that $x_S$ is the boundary divisor associated to $S$ and $h_S$ is defined via these boundary divisors by \eqref{eq:hS}, or equivalently (via \cref{lem:hS-pullback-psi}) it is given by $h_S = F_S^*\psi_y$.

We begin with a lemma that follows directly from the relations in $A^*(\Sigma^r_n)\cong A^*(\Lrn)$.  Here, for every $S \in \decoset$, we denote by $\underline{S}$ the image of $S$ in $[n]$---in other words, the set obtained from $S$ by forgetting the superscripts.

\begin{lemma} \label{lem:xTtimeshS} Let $k \in [n]$, and let $S \in \decoset$ be such that $k \in \underline{S}$.  Then
\[h_{\{k^j\}} \cdot x_S = 0\]
for all $j \in \Z_r$.
\end{lemma}
\begin{proof} By assumption we have $k^i \in S$ for some $i\in \Z_r$, and without loss of generality, we can assume $i\neq j$, since the linear relation \eqref{eq:linear} in $A^*(\Lrn)$ implies $h_{\{k^j\}}=h_{\{k^{j'}\}}$.  Thus, we have $k^j \notin S$, so the quadratic relations \eqref{eq:quadratic} for $A^*(\Lrn)$ imply that $x_{S'} x_S = 0$ for all $S' \in \decoset$ containing $k^j$.  Since
\[h_{\{k^j\}} = \sum_{k^j \in S'} x_{S'},\]
it follows that $h_{\{k^j\}}x_S=0$, as claimed.
\end{proof} 

This lemma already allows us to give a geometric proof of \cref{thm:h-integrals} in the case where the sets $S_1, \ldots, S_n \in \decoset$ define a maximal chain.

\begin{proposition}
\label{prop:strict}
    Let $(S_1 \subsetneq S_2 \subsetneq \dots \subsetneq S_n) \in \maxchain$.     Then
    \[\int_{\L^r_n} h_{S_1} h_{S_2} \cdots h_{S_n} = 1,\]
    so in particular, \cref{thm:h-integrals} holds in this case.
\end{proposition}
\begin{proof}
    The proof is by induction on $n$.  As a base case, suppose that $n=1$.  Then the integral in question is
    \[\int_{\L^r_1} h_S,\]
    in which $S$ is a singleton.  But $h_S = x_S$ when $n=1$, so we have
    \[\int_{\L^r_1} h_S = \int_{\L^r_1} x_S,\]
    which equals $1$ because $\L^r_1 \cong \P^1$ and the boundary divisor $x_S$ is the class of a single point.

    Now, suppose that the lemma holds on $\L^r_{n-1}$.  The hypothesis that the chain is maximal implies that $S_1 = \{k^j\}$ for some $k^j \in E$, and we claim that, if $S_i'\coloneqq  S_i \setminus \{k^j\}$ for each $i \in [n]$, then
    \begin{equation}
        \label{eq:induction}
   \int_{\L^r_n} h_{S_1} \cdots h_{S_n} = \int_{\L^r_{n-1}} h_{S_2'} \cdots h_{S_n'}.
    \end{equation}
    If we can prove this, then the proposition will follow by the induction hypothesis.

    To prove \eqref{eq:induction}, we first note that, whenever $S \in \decoset$ is such that $k \in \underline{S}$, \cref{lem:xTtimeshS} implies
    \[h_{S_1} x_S = h_{\{k^j\}} \cdot x_S = 0.\]
    It follows that, in the product $h_{S_1}h_{S_2} \cdots h_{S_n}$, one can replace $h_{S_i}$ for each $i\in \{2, \ldots, n\}$ by 
    \[\sum_{\substack{S \cap S_i \neq \emptyset \\ k \notin \underline{S}}} x_{S} =\phi^*\left(h_{S_i'}\right),\]
    where $\phi \colon \Lrn \rightarrow \L^r_{n-1}$ is the forgetful map forgetting the light orbit indexed by $k$.  That is,
    \[h_{S_1} h_{S_2} \cdots h_{S_n} = h_{S_1} \phi^*\big(h_{S_2'} \cdots h_{S_n'}\big).\]
    It therefore follows from the projection formula that 
    \[\int_{\L^r_n} h_{S_1} \cdots h_{S_n} = \int_{\L^r_{n-1}} \phi_*(h_{S_1}) h_{S_2'} \cdots h_{S_n'},\]
    so to prove \eqref{eq:induction}, it suffices to prove that $\phi_*(h_{S_1}) = 1$.  To see this, express
\[h_{S_1} = x_{S_1} + \sum_{\substack{S\supseteq S_1\\ |S|\ge 2}} x_S.\]
Then
\[\phi_*(h_{S_1}) = (\phi|_{X_{S_1}})_*\left(\mathbf{1}_{X_{S_1}}\right) + \sum_{\substack{S\supseteq S_1\\ |S|\ge 2}}(\phi|_{X_S})_*\left(\mathbf{1}_{X_T}\right),\]
 where we recall that $X_S \subseteq \Lrn$ is the subvariety such that $[X_S] = x_S \in A^1(\Lrn)$.  Geometrically, one sees that $\phi|_{X_{S_1}} \colon X_{S_1} \rightarrow \L^r_{n-1}$ is an isomorphism, so $(\phi|_{X_{S_1}})_*\left(\mathbf{1}_{X_{S_1}}\right) =1$ (see, for example, \cite[Proposition 5.5]{clader2022permutohedral}).  On the other hand, for each $S$ appearing in the summation above, $\phi|_{X_S} \colon X_S \rightarrow \L^r_{n-1}$ reduces the dimension, so $(\pi|_{X_S})_*\left(\mathbf{1}_{X_S}\right) = 0$.  Thus, we indeed have $\phi_*\left(h_{S_1}\right) = 1$, so the proposition is proved. 
\end{proof}

\begin{remark}
    In fact, the proof of \cref{prop:strict} holds as long as the underlying sets 
 are strictly nested (that is, $\underline{S}_1 \subsetneq \cdots \subsetneq \underline{S}_n$), without necessarily assuming the stronger condition that ${S}_1 \subsetneq \cdots \subsetneq {S}_n$. 
\end{remark}

\begin{remark}
    The equation \eqref{eq:induction} can also be used to prove that \cref{thm:h-integrals} holds in the case where $S_1, \ldots, S_n$ are pairwise disjoint.  The main idea of the argument is to use the relations in $A^*(\Lrn)$ to argue that
    \[\int_{\Lrn} h_{S_1} \cdots h_{S_{n-1}} h_{S_n} = \sum_{a \in S_n} \int_{\Lrn} h_{S_1} \cdots h_{S_{n-1}} h_{\{a\}}\]
    by applying the pairwise disjoint hypothesis, and then to apply \eqref{eq:induction} to rewrite this as
    \[\sum_{a \in S_n} \int_{\L^r_{n-1}} h_{S_1 \setminus \{a\}} \cdots h_{S_n \setminus \{a\}}.\]
    This sets up an induction on $n$ from which \cref{thm:h-integrals} immediately follows.
\end{remark}

Building off of \cref{prop:strict}, one can prove \cref{thm:h-integrals} for collections $S_1, \ldots, S_n \in \decoset$ that are nested but not strictly nested.  The key ingredient in the proof of this generalization is the following geometric lemma.

\begin{lemma}
\label{lem:hSsquared}
Let $S \in \decoset$, and choose any element $\star \in S$.  Then
    \[(h_S)^2 = \begin{cases} h_Sh_{S \setminus \{\star\}} & \text{ if } |S|>1,\\ 0 & \text{ if } |S|=1\end{cases}.\]
\end{lemma}
\begin{proof}
When $|S|=1$, the result follows from \cref{lem:xTtimeshS}. Now, suppose that $|S| >1$.  Choose any element $\star \in S$, and set
\[S'\coloneqq  S \setminus \{\star\}.\]
By \cref{lem:hS-pullback-psi}, we have $h_S = F_S^*(\psi_y)$ and $h_{S'} = F_{S'}^*(\psi_y)$.
Furthermore, if $f: \Mbar^1_S \rightarrow \Mbar^1_{S'}$ is the map forgetting the light marked point $z_\star$, we have $F_{S'} = f \circ F_S$.  Applying \eqref{eq:psiycomparison} then shows
\begin{align*}
    h_Sh_{S'} &= F_S^*(\psi_y)F_S^*(f^*(\psi_y))\\
    &=F_S^*(\psi_y) F_S^*(\psi_y-D^1_{\{\star\}})\\
    &=F_S^*(\psi_y)^2 - F_S^*(\psi_y \cdot D^1_{\{\star\}})\\
    &=(h_S)^2 - F_S^*(\psi_y \cdot D^1_{\{\star\}}).
\end{align*}
However, we have $\psi_y \cdot D^1_{\{\star\}} = 0$, because in the divisor $D^1_{\{\star\}}$, the marked point $y$ lies on a genus-zero component with only three special points, so its cotangent line bundle is trivial.  Therefore, $h_Sh_{S'} = (h_S)^2$, as claimed.
\end{proof}
\begin{proposition}
\label{prop:nonstrictlynested}
Let $S_1, \ldots, S_n \in \decoseto$ be such that $S_1 \subseteq S_2 \subseteq \cdots \subseteq S_n$.  Then
    \[\int_{\L^r_n} h_{S_1} h_{S_2} \cdots h_{S_n} = \begin{cases} 1 & \text{ if } |S_i| \geq i \text{ for all } i \in [n],\\ 0 & \text{ otherwise}.\end{cases}\]
In particular, \cref{thm:h-integrals} holds in this case.
\end{proposition}

\begin{proof}  First, we prove that the integral equals zero whenever $|S_i| < i$ for some $i \in [n]$; the proof is by induction on the smallest $i$ for which this occurs.  Since each $S_i$ has size at least $1$, the base case is $i=2$: that is, we suppose that $|S_2| <2$.  This means that $S_1 = S_2 = \{\star\}$ for some $\star \in E$, but then 
\[h_{S_1} h_{S_2} = h_{\{\star\}} h_{\{\star\}} = 0\]
by \cref{lem:hSsquared}.

Now, fix $i > 1$, and suppose that the integral equals zero for all chains
\[\scC'= (S_1' \subseteq S_2' \subseteq \cdots \subseteq S_n')\]
of colored sets such that $|S'_{i-1}| < i-1$.  Fix a chain
\[\scC= (S_1 \subseteq S_2 \subseteq \cdots \subseteq S_n)\]
with $|S_{i-1}| = i-1$ but $|S_i|<i$.  This forces that $S_{i-1} = S_i$, and so there exists some $k<i$ with
\[S_{k-1} \subsetneq S_k = S_{k+1} = \cdots = S_{i-1} = S_i.\]
Now, choose any $\star \in S_{k-1} \setminus S_{k}$, and for each $j \in \{k, \ldots, i-1\}$, set $S_{j}'\coloneqq  S_{j} \setminus \{\star\}$.  Then
\[h_{S_{j}}h_{S_i} = h_{S'_{j}}h_{S_i}\]
by \cref{lem:hSsquared}.  It follows that we can replace the chain $\scC$ with the chain
\[\scC'\coloneqq  (S_1 \subseteq \cdots \subseteq S_{k-1} \subseteq S'_k \subseteq \cdots \subseteq S'_{i-1} \subseteq S_i \subseteq \cdots \subseteq S_n)\]
without affecting the integral in question, but the integral for the chain $\scC'$ equals zero by the induction hypothesis.

We have therefore proven that the integral equals zero unless $|S_i| \geq i$ for all $i$, and what remains to be shown is that it equals $1$ when this condition is satisfied.  This proof is again by induction, this time on the number
    \[r(\scC) \coloneqq  \left|\left\{i \in \{1, \ldots, n-1\} \; \big| \; S_{i} = S_{i+1}\right\}\right|\]
of repetitions in $\scC$.

If $r(\scC) = 0$, then $\scC$ is strictly nested and the statement follows from \cref{prop:strict}.  Suppose, then, that $\scC = (S_1 \subseteq \cdots \subseteq S_n)$ is a chain with at least one repetition and that the proposition holds for all chains $\scC'$ with $r(\scC') < r(\scC)$.

Let $i$ be the minimum index such that $S_{i} = S_{i+1}$.  It follows that $S_{i-1} \subsetneq S_{i}$, so there exists $x \in S_i \setminus S_{i-1}$.  Let
\[T_i\coloneqq  S_i \setminus \{x\},\]
which is nonempty by the condition $|S_{i+1}| \geq i+1 \geq 2$.  Then
\[h_{S_{i}}h_{S_{i+1}} = h_{T_{i}}h_{S_{i+1}}.\]
by \cref{lem:hSsquared}, so we can replace the chain $\scC$ by the chain
\[\scC' \coloneqq  (S_1 \subseteq \cdots \subseteq S_{i-1} \subseteq T_i \subseteq S_{i+1} \subseteq \cdots \subseteq S_n)\]
without affecting the integral in question.  As long as $S_{i-1} \subsetneq T_i$, we have $r(\scC') < r(\scC)$ and therefore the integral equals $1$ by the induction hypothesis.  If $S_{i-1} = T_i$, then we we repeat the argument, replacing $S_{i-1}$ by $T_{i-1} \coloneqq  S_{i-1} \setminus \{y\}$ for $y \in S_{i-1} \setminus S_{i-2}$.  This process eventually terminates, so the integral equals $1$ by the inductive hypothesis, completing the proof. \end{proof}

\bibliographystyle{alpha}
\bibliography{bibliography.bib}

\end{document}